\documentclass[10pt,a4paper]{smfart}

\usepackage[utf8]{inputenc}
\usepackage[french, english]{babel}
\usepackage[T1]{fontenc}
\usepackage{lmodern}

\usepackage[hidelinks,pdfencoding=unicode]{hyperref}
\usepackage[shortcuts]{extdash}
\usepackage{xpatch}

\usepackage{amsmath}
\usepackage{amssymb}
\usepackage{stmaryrd}
\usepackage{mathtools} 

\usepackage[all,2cell]{xy}
\UseAllTwocells
\SilentMatrices
\SelectTips{cm}{10}

\usepackage{version}

\usepackage{enumitem}
\setlist{nosep}
\setlist[enumerate, 1]{label={\rm (}\emph{\alph*}{\rm )}}
\setlist[enumerate, 2]{label={\rm (}\emph{\alph{enumi}.\arabic*}{\rm )}}

\makeatletter

\patchcmd\@maketitlehook
  {\ifx\@empty\@keywords\else\@footnotetext{\@setkeywords}\fi}
  {\ifx\@empty\@keywords\else\@footnotetext{\@setkeywords}\fi
   \ifx\@empty\@altkeywords\else\@footnotetext{\@setaltkeywords}\fi}
  {}{}

\makeatother

\makeatletter
\patchcmd{\smf@captionsfrench}{2000}{2020}{}{}
\makeatother

\newcommand\nohyphen\hbox

\newif\iffr
\frtrue

\newcommand\fren[2]{\iffr #1\else #2\fi}

\newcommand\journal\emph

\renewcommand\and{\fren{et}{and}\xspace}

\theoremstyle{plain}
\newtheorem{thm}{Th\'eor\`eme}[section]
\newtheorem*{thm*}{Th\'eor\`eme}

\newtheorem{lemma}[thm]{Lemme}
\newtheorem{coro}[thm]{Corollaire}
\newtheorem*{coro*}{Corollaire}
\swapnumbers
\newtheorem{lemmamulti}[thm]{Lemme multi-simplicial}
\newtheorem{lemma_bi_simpl}[thm]{Lemme bisimplicial}
\swapnumbers
\theoremstyle{remark}

\theoremstyle{definition}

\newtheorem*{notations}{Notations et terminologie}

\newtheorem*{organisation}{Organisation de l'article}
\theoremstyle{plain}
\swapnumbers

\swapnumbers

\numberwithin{equation}{thm}

\makeatletter
\newif\ifsection
\preto\section{\sectiontrue}
\preto\subsection{\sectionfalse}
\xapptocmd\@sect{%
  \ifsection
    \numberwithin{thm}{section}
  \else
    \numberwithin{thm}{subsection}
  \fi
  \setcounter{thm}{0}\relax}
  {}{}
\makeatother

\newcommand\letenv[2]{%
\expandafter\expandafter\expandafter\let\expandafter\expandafter
\csname #1\endcsname\csname #2\endcsname
\expandafter\expandafter\expandafter\let\expandafter\expandafter
\csname end#1\endcsname\csname end#2\endcsname
}

\letenv{paragraph}{paragr}
\letenv{proposition}{prop}
\letenv{corollary}{coro}
\letenv{theorem}{thm}
\letenv{theorem*}{thm*}
\letenv{corollary*}{coro*}
\letenv{remark}{rem}
\letenv{example}{exem}

\let\forlang\emph
\let\ndef\emph
\let\nbd\nobreakdash

\def\xpoint{\futurelet\@let@token\@xpoint}
\def\@xpoint{%
  \ifx\@let@token.\else
    .%
  \fi
  \xspace}

\newcommand\resp[1]{{\rm (resp.}~#1{\rm )}}

\newcommand\zbox[1]{\makebox[0pt][l]{#1}}
\newcommand\pbox[1]{\zbox{\,\,#1}}
\newcommand\bpbox[1]{\zbox{\quad#1}}

\newcommand\quadtext[1]{\quad\text{#1}\quad}

\newcommand\quadet{\quadtext{et}}

\newcommand\taut{\widetilde{\tau}}

\renewcommand\le\leqslant
\renewcommand\ge\geqslant
\renewcommand\epsilon\varepsilon
\renewcommand\phi\varphi

\let\ot\leftarrow

\let\hookto\hookrightarrow
\let\xto\xrightarrow
\let\xot\xleftarrow
\newcommand\tod\Rightarrow
\newcommand\otd\Leftarrow
\newcommand\tot\Rrightarrow

\newcommand\e{\epsilon}

\newcommand{\sauf}{\mathchoice{\raise 1.8pt\hbox{${\scriptstyle\kern
2.5pt\smallsetminus\kern 2.5pt}$}}{\raise 1.8pt\hbox{${\scriptstyle\kern
2.5pt\smallsetminus\kern 2.5pt}$}}{\raise
1.8pt\hbox{${\scriptscriptstyle\kern 1.5pt\smallsetminus\kern
1.5pt}$}}{\raise 1.8pt\hbox{${\scriptscriptstyle\kern
1.5pt\smallsetminus\kern 1.5pt}$}}}

\newcommand\HomOpLax{\Homi_{\mathrm{oplax}}}
\newcommand\HomLax{\Homi_{\mathrm{lax}}}

\newcommand\Z{\mathbb{Z}}
\newcommand\N{\mathbb{N}}

\let\limind\varinjlim

\let\C\relax
\newcommand\C{\mathcal{C}}
\newcommand\M{\mathcal{M}}

\newcommand{\Hom}{\operatorname{\mathsf{Hom}}}
\newcommand{\Homi}{\operatorname{\kern.5truept\underline{\kern-.5truept\mathsf{Hom}\kern-.5truept}\kern1truept}}

\newcommand{\pref}[1]{{\widehat{ #1 }}}
\newcommand\id[1]{1_{#1}}
\renewcommand\o\circ

\newcommand{\Ob}{\operatorname{\mathsf{Ob}}}

\newcommand{\Ens}{{\mathcal{E}\mspace{-2.mu}\it{ns}}}

\newcommand{\Cat}{{\mathcal{C}\mspace{-2.mu}\it{at}}}
\newcommand{\Ho}{\operatorname{Ho}}
\newcommand{\Hot}{{\mathsf{Hot}}}
\newcommand{\nCat}[1]{%
\mathchoice
  {\hbox{$#1$\kern1pt-\kern1pt$\Cat$}}
  {\hbox{$#1$\kern1pt-\kern1pt$\Cat$}}
  {\hbox{$\scriptstyle #1$\raisebox{-0.5pt}{-}$\scriptstyle \Cat$}}
  {TODO}%
}
\let\CCat\nCat
\newcommand{\nGpd}[1]{%
\mathchoice
  {\hbox{$#1$-\kern1pt$\Gpd$}}
  {\hbox{$#1$-\kern1pt$\Gpd$}}
  {\hbox{$\scriptstyle #1$\raisebox{-0.5pt}{-}$\scriptstyle \Gpd$}}
  {TODO}%
}

\newcommand{\ooCat}{\nCat{\infty}}
\newcommand{\coCat}{{\it{co}\mathcal{C}\mspace{-2.mu}\it{at}}}
\newcommand{\ncoCat}[1]{{#1}\hbox{\protect\nbd-}\kern1pt\coCat}

\newcommand\oo[1]{$\infty$\=/}

\newcommand\ponct{e}

\newcommand\Dn[1]{\mathrm{D}_{#1}}

\newcommand\cDelta{\mathbf{\Delta}}

\newcommand\Deltan[1]{\varDelta_{#1}}

\newcommand\cO{\mathcal{O}}

\newcommand\On[1]{\mathcal{O}_{#1}}

\newcommand{\tr}[2]{\mathchoice
  {#1\raise -1.8pt\vbox{\hbox{$\kern -.8pt/#2$}}}
  {#1\raise -1.8pt\vbox{\hbox{$\kern -.8pt/#2$}}\kern .8pt}
  {#1\raise -1.8pt\vbox{\hbox{$\scriptstyle\kern -.8pt /#2$}}}
  {#1\raise -1.8pt\vbox{\hbox{$\scriptscriptstyle\kern -.8pt /#2$}}}}

\newcommand{\trm}[2]{\mathchoice
  {#1\raise -1.8pt\vbox{\hbox{$\kern -.8pt\!\stackrel{\,\rm co}{/}\!\!#2$}}}
  {#1\raise -1.8pt\vbox{\hbox{$\kern -.8pt\!\stackrel{\,\rm co}{/}\!\!#2$}}\kern .8pt}
  {#1\raise -1.8pt\vbox{\hbox{$\scriptstyle\kern -.8pt\!\stackrel{\,\,\rm co}{/}\!\!#2$}}\kern .8pt}
  {TODO}}

\newcommand{\trto}[2]{\mathchoice
  {#1\raise -1.8pt\vbox{\hbox{$\kern -.8pt\!\stackrel{\,\,t \rm o}{/}\!\!\!#2$}}}
  {#1\raise -1.8pt\vbox{\hbox{$\kern -.8pt\!\stackrel{\,\,t \rm o}{/}\!\!\!#2$}}\kern .8pt}
  {#1\raise -1.8pt\vbox{\hbox{$\scriptstyle\kern -.8pt\!\stackrel{\,\,t \rm o}{/}\!\!\!#2$}}\kern .8pt}
  {TODO}}

\newcommand{\cotr}[2]{\mathchoice
  {\raise -1.8pt\vbox{\hbox{$#2\backslash$}}#1}
  {\raise -1.8pt\vbox{\hbox{$#2\backslash$}}#1}
  {\raise -1.8pt\vbox{\hbox{$\scriptstyle#2\backslash$}}#1}
  {\raise -1.8pt\vbox{\hbox{$\scriptscriptstyle#2\backslash$}}#1}}

\newcommand{\cotrm}[2]{\mathchoice
  {\raise -1.8pt\vbox{\hbox{$#2\!\stackrel{\!\rm co}{\backslash}$}}#1}
  {\raise -1.8pt\vbox{\hbox{$#2\!\stackrel{\!\rm co}{\backslash}$}}#1}
  {\raise -1.8pt\vbox{\hbox{$\scriptstyle#2\!\stackrel{\!\rm co}{\backslash}$}}#1}
  {TODO}}

\newcommand\Cyl\Gamma

\newcommand\NS{N}
\newcommand\Nsn[1]{N_{\cDelta^{\mkern-2mu #1}}}
\newcommand\NC{N_\Theta}
\newcommand\NCn[1]{N_{\Theta_{#1}}}

\newcommand\wrDelta[1]{\cDelta \wr #1}

\newcommand\otimest{\oslash}
\newcommand\Homidt{\mathop{\Homi^\otimest_{\mathrm{oplax}}}}
\newcommand\Homigt{\mathop{\Homi^\otimest_{\mathrm{lax}}}}

\newcommand{\Cda}{\mathcal{C}_{\mathrm{da}}}

\newcommand\Zdec{\underline{\mathbb{Z'}\kern -2.5pt}\kern 2pt}

\newcommand{\pushoutcorner}[1][dr]{\save*!/#1+1.5pc/#1:(1,-1)@^{|-}\restore}
\newcommand{\pullbackcorner}[1][dr]{\save*!/#1-1.5pc/#1:(-1,1)@^{|-}\restore}

\newcommand\W{\mathcal{W}}

\author{Dimitri Ara}
\address{Aix~Marseille~Univ,~CNRS,~I2M,~Marseille,~France}
\email{dimitri.ara@univ-amu.fr}
\urladdr{\href{http://www.i2m.univ-amu.fr/perso/dimitri.ara/}{http://www.i2m.univ-amu.fr/perso/dimitri.ara/}}

\author{Georges Maltsiniotis}
\address{%
Université Paris Cité, CNRS, IMJ, Paris, France}
\email{georges.maltsiniotis@imj-prg.fr}
\urladdr{\href{https://webusers.imj-prg.fr/~georges.maltsiniotis/}
{https://webusers.imj-prg.fr/%
\raise -3.3pt\vbox{\hbox{$\widetilde{ \ }\,$}}georges.maltsiniotis/}}

\title{Comparaison des nerfs $n$-catégoriques}
\alttitle{Comparison of the $n$-categorical nerves}

\dedicatory{À Ondine, de l'automnal résultat bisimplicial à l'éternité. D.}

\begin{document}

\frontmatter

\begin{abstract}
  Le but de cet article est de comparer trois foncteurs nerf pour les
  $n$\=/catégories strictes : le nerf de Street, le nerf cellulaire et le nerf
  multi-simplicial. Nous montrons que ces trois foncteurs sont équivalents
  en un sens adéquat. En particulier, les classes d'équivalences faibles
  $n$-catégoriques qu'ils définissent coïncident : ce sont les équivalences
  de Thomason. On donne deux applications de ce résultat : la première
  affirme qu'une équivalence de type Dwyer-Kan pour les équivalences de
  Thomason est une équivalence de Thomason ; la seconde, fondamentale, est
  la stabilité de la classe des équivalences de Thomason par les dualités de la
  catégorie des $n$-catégories strictes.
\end{abstract}

\begin{altabstract}
  Our aim is to compare three nerve functors for strict $n$\=/categories:
  the Street nerve, the cellular nerve and the multi-simplicial nerve. We
  show that these three functors are equivalent in some appropriate sense.
  In particular, the classes of $n$-categorical weak equivalences that they
  define coincide: they are the Thomason equivalences. We give two
  applications of this result: the first one states that a Dyer-Kan-type
  equivalence for Thomason equivalences is a Thomason equivalence; the
  second one, fundamental, is the stability of the class of Thomason
  equivalences under the dualities of the category of strict $n$-categories.
\end{altabstract}


\altkeywords{strict \oo-categories, cellular sets, simplicial sets, Thomason
equivalences, nerve functors, cellular nerve, Street nerve, multi-simplicial
nerve, Gray tensor product, lax transformations}

\maketitle


\mainmatter


\section*{Introduction}

Ce travail s'inscrit dans une série d'articles sur la théorie de l'homotopie
des \oo-catégories strictes, série constituée actuellement de nos articles
\cite{AraMaltsiNThom}, \cite{Ara2Cat}, \cite{AraMaltsiJoint},
\cite{AraMaltsiThmAI}, \cite{AraMaltsiThmAII}, \cite{AraThmB},
\cite{AraMaltsiCondE} et de celui de Gagna~\cite{Gagna}. Cette théorie de
l'homotopie des \oo-catégories strictes est inspirée de la théorie de
l'homotopie des petites catégories de Grothendieck, introduite dans
\cite{GrothPS} et développée dans \cite{Maltsi} et \cite{Cisinski}, qui
consiste en l'étude de $\Cat$, la catégorie des petites catégories, munie de
la classe des équivalences de Thomason, c'est-à-dire des foncteurs dont
l'image par le foncteur nerf $N : \Cat \to \pref{\cDelta}$, à valeurs dans
les ensembles simpliciaux, est une équivalence d'homotopie faible. Un
résultat fondamental de Quillen, rédigé par Illusie dans sa
thèse~\cite{Illusie}, affirme que le foncteur nerf induit une équivalence
entre la localisée de $\Cat$ par rapport aux équivalences de Thomason et la
catégorie homotopique des ensembles simpliciaux, c'est-à-dire la catégorie
des types d'homotopie.  Ainsi, étudier les petites catégories munies des
équivalences de Thomason, c'est étudier les types d'homotopie sous un jour
nouveau.

Un des désavantages de ce modèle des petites catégories pour représenter les types
d'homotopie est que les petites catégories représentant un type d'homotopie
donné sont souvent peu naturelles et de nature peu géométrique. C'est pourquoi nous
proposons dans le projet dont est issu cet article de remplacer les
catégories, par nature $1$-dimensionnelles, par les \oo-catégories strictes.
Nous expliquons longuement nos motivations pour ce modèle dans
l'introduction de \cite{AraMaltsiThmAI}.

Afin de généraliser la classe des équivalences de Thomason de $\Cat$ à
$\ooCat$, la catégorie des \oo-catégories strictes et des \oo-foncteurs
stricts, on a \forlang{a priori} besoin de disposer d'un foncteur nerf de
source $\ooCat$. Or, on dispose d'au moins deux foncteurs nerf naturels pour
les \oo-catégories strictes :
\begin{enumerate}
  \item Le nerf de Street
  \[ \NS : \ooCat \to \pref{\cDelta}, \]
  foncteur nerf à valeurs dans les ensembles simpliciaux, introduit par
  Street dans~\cite{StreetOrient}, basé sur son objet cosimplicial $\cO :
  \cDelta \to \ooCat$ des orientaux, version \oo-catégorique des simplexes.
  \item Le nerf cellulaire
  \[ \NC : \ooCat \to \pref{\Theta}, \]
  à valeurs dans les ensembles cellulaires, c'est-à-dire les préfaisceaux
  sur la catégorie $\Theta$ de Joyal \cite{JoyalTheta},
  qui est induit par l'inclusion canonique $\Theta \hookto \ooCat$ et est
  pleinement fidèle.
\end{enumerate}
De plus, lorsque l'on se restreint, pour $n \ge 0$, à $\nCat{n}$, la
sous-catégorie pleine de~$\ooCat$ formée des $n$-catégories strictes, on
dispose d'un troisième foncteur nerf naturel :
\begin{enumerate}[resume]
  \item Le nerf $n$-simplicial (ou multi-simplicial, si on ne veut pas
  préciser $n$)
  \[ \Nsn{n} : \nCat{n} \to \pref{\cDelta^n}, \]
  à valeurs dans les ensembles $n$-simpliciaux,
  qui s'obtient par un usage itéré du nerf usuel.
\end{enumerate}
\goodbreak
Dans nos précédents travaux, nous avons privilégié le choix du nerf de
Street en définissant les équivalences de Thomason \oo-catégoriques comme les
\oo-foncteurs dont le nerf de Street est une équivalence d'homotopie faible
simpliciale. Néanmoins, le nerf cellulaire permet également de
définir une notion d'équivalence faible \oo-catégorique, et de même pour le
nerf multi-simplicial, du moins lorsque l'on se restreint aux $n$\=/catégories
pour $n$ fini.

\bigbreak

Le but du présent article est de montrer que les notions d'équivalences
faibles définies par ces trois foncteurs nerf coïncident,
en se restreignant aux $n$-catégories pour le troisième.  Plus
précisément, on établit que ces trois foncteurs nerf sont équivalents,
toujours avec la même restriction pour le troisième, en un sens que
l'on va maintenant préciser.

Si $A$ est une petite catégorie, on notera $\pref{A}$ la catégorie des
préfaisceaux sur $A$ et $i_A : \pref{A} \to \Cat$ le foncteur qui associe à
un préfaisceau $F$ sur $A$ sa catégorie des éléments~$\tr{A}{F}$. En
composant par le foncteur nerf, suivi du foncteur canonique~$p :
\pref{\cDelta} \to \Hot$ vers la catégorie homotopique des ensembles
simpliciaux, on obtient un foncteur $k_A : \pref{A} \to \Hot$ associant à
tout préfaisceau sur $A$ un type d'homotopie. Dans le cas $A = \cDelta$, en
vertu d'un théorème de Quillen, le foncteur $k_\cDelta$ est canoniquement
isomorphe au foncteur~$p$.

\goodbreak

On peut maintenant énoncer nos principaux résultats :
\begin{theorem*}
  Le nerf de Street et le nerf cellulaire sont équivalents au sens où les
  foncteurs composés
  \[
    \ooCat \xto{\NS} \pref{\cDelta} \xto{k_{\cDelta}} \Hot
    \quadet
    \ooCat \xto{\NC} \pref{\Theta} \xto{k_{\Theta}} \Hot
  \]
  sont isomorphes.
\end{theorem*}

\begin{corollary*}
  Le nerf de Street et le nerf cellulaire définissent les mêmes équivalences
  \oo-catégoriques au sens où, si $u : A \to B$ est un \oo-foncteur, alors
  $k_\cDelta \NS u$ est un isomorphisme si et seulement si $k_\Theta \NC u$
  en est un.
\end{corollary*}

Les \oo-foncteurs $u$ vérifiant les conditions équivalentes de ce corollaire
sont les \ndef{équivalences de Thomason} \oo-catégoriques.

\begin{theorem*}
  Le nerf $n$-simplicial et le nerf de Street (et donc le nerf
  cellulaire) restreint aux $n$-catégories strictes sont équivalents au
  sens où les foncteurs composés
  \[
    \nCat{n} \xto{\Nsn{n}} \pref{\cDelta^n} \xto{k_{\cDelta^{\mkern-2mu n}}} \Hot
    \quadet
    \nCat{n} \hookto \ooCat \xto{\NS} \pref{\cDelta} \xto{k_{\cDelta}} \Hot
  \]
  sont isomorphes.
\end{theorem*}

\begin{corollary*}
  Pour tout $n \ge 0$, le nerf $n$-simplicial et le nerf de Street (et donc le
  nerf cellulaire) définissent les mêmes équivalences faibles
  $n$-catégoriques.
\end{corollary*}

Dans le cas $n = 2$, ce dernier théorème et son corollaire (sans l'addendum
sur le nerf cellulaire) sont dus à Bullejos et Cegarra \cite{BullCegGeom2Cat}.
De plus, Cegarra a généralisé ces résultats aux bicatégories avec Carrasco
et Garzón~\cite{CegNervesBicat} et aux tricatégories avec
Heredia~\cite{CegRealTri}.

\bigbreak

Disons à présent un mot sur la démonstration de ces résultats, qui est en
partie inspirée de \cite{BullCegGeom2Cat}. À une \oo-catégorie~$C$, on
associe un objet simplicial $SC$ dans $\ooCat$
dont la \oo-catégorie des $p$-simplexes est
  \[
    S_pC = \coprod_{c_0, \dots, c_p \in \Ob(C)} \Homi_C(c_0, c_1) \times \cdots
      \times \Homi_C(c_{p-1}, c_p).
  \]
En appliquant le nerf de Street argument par argument, on obtient un
ensemble bisimplicial $\NS SC$. On montre, et c'est là le cœur de ce
travail, que la diagonale de cet ensemble bisimplicial est faiblement
équivalente au nerf de Street $\NS C$ de $C$. Pour ce faire, on interprète
les $(p, q)$-simplexes de $\NS SC$ comme les \oo-foncteurs de~$\On{q}
\otimest \Deltan{p}$ vers~$C$, où $\On{q}$ est le $q$-ième oriental,
$\Deltan{p}$ la catégorie correspondant à l'ensemble ordonné $\{0 < \cdots <
p\}$ et $\otimest$ est un quotient du produit de Gray $\otimes$, qu'on
introduit dans cet article.  Les \oo-foncteurs canoniques $\On{q} \to
\Deltan{0}$ et~$\On{p} \to \Deltan{p}$ induisent des applications
\[
  \NS SC_{p,q} \simeq \Hom_{\ooCat}(\On{q} \otimest \Deltan{p}, C)
  \xto{}
  \Hom_{\ooCat}(\On{q} \otimest \On{p}, C)
  \xot{}
  \Hom_{\ooCat}(\Deltan{0} \otimest \On{p}, C).
\]
Or, un isomorphisme canonique $\Deltan{0} \otimest \On{p} \simeq \On{p}$
permet d'identifier l'ensemble de droite avec l'ensemble des
$p$-simplexes de $\NS C$, et on obtient l'équivalence mentionnée,
grâce à des résultats que nous établissons sur la construction $\otimest$,
en utilisant le fait que les deux \oo-foncteurs canoniques en jeu sont des
rétractions de rétracte par transformation, lax pour l'un et oplax pour
l'autre.

Une application répétée de ce résultat permet alors d'obtenir la comparaison du
nerf de Street $n$-catégorique et du nerf $n$-simplicial.

En ce qui concerne la comparaison du nerf de Street et du nerf cellulaire,
on passe par l'intermédiaire d'une comparaison du nerf $n$-cellulaire et du
nerf $n$-simplicial. Pour cela, on montre que le foncteur canonique
$\cDelta^n \to \Theta_n$, la catégorie $\Theta_n$ étant la variante
$n$-dimensionnelle de la catégorie~$\Theta$, est asphérique, au sens où il
vérifie les hypothèses du théorème A de Quillen, ce qui, par un résultat de
Grothendieck, entraîne cette comparaison intermédiaire. Par transitivité, on
obtient donc une comparaison entre le nerf de Street $n$-catégorique et le
nerf $n$-cellulaire. Le passage du cas $n$-catégorique au cas
\oo-catégorique est ensuite essentiellement formel, grâce à un résultat
général de comparaison de nerfs, basé sur la théorie des structures
d'asphéricité de Grothendieck développée dans \cite{GrothPS}.

\bigbreak

Pour finir, on présente deux applications de ces résultats de comparaison.
La première, à propos des équivalences de type Dwyer-Kan pour les
équivalences de Thomason, se prouve en utilisant notre résultat sur l'objet
simplicial en \oo-catégories strictes $SC$ :
\begin{theorem*}
  Soit $u : C \to D$ un \oo-foncteur. On suppose que
  \begin{enumerate}
    \item $u$ induit une bijection $\Ob(C) \xto{\sim} \Ob(D)$ ;
    \item pour tous objets $c$ et $c'$ de $C$, le \oo-foncteur
        \[ \Homi_C(c, c') \to \Homi_D(u(c), u(c')) \]
        induit par $u$ est une équivalence de Thomason.
  \end{enumerate}
  Alors $u$ est une équivalence de Thomason.
\end{theorem*}

La seconde est la stabilité des équivalences de Thomason par dualité,
stabilité qui s'obtient facilement dans le cas du nerf cellulaire et qu'on
peut donc transposer au nerf de Street :
\begin{theorem*}
  Soit $J$ un sous-ensemble de $\N\sauf\{0\}$
  et notons $D_J : \ooCat \to \ooCat$ le foncteur qui envoie une
  \oo-catégorie $C$ sur la \oo-catégorie obtenue à partir de $C$ en inversant
  l'orientation des $j$-cellules pour $j$ dans $J$.
  Alors un \oo-foncteur $u$ est une équivalence de Thomason si et seulement si
  $D_J(u)$ en est une.
\end{theorem*}
Cet énoncé, fondamental, montre que la classe des équivalences de Thomason
est intrinsèque à la catégorie $\ooCat$ au sens où elle est invariante par
tous les automorphismes de $\ooCat$ et ne dépend donc que de sa seule
structure de catégorie.

On tire des conséquences de ce résultat qui avaient été annoncées dans
des articles précédents. Dans \cite{AraMaltsiThmAI} et
\cite{AraMaltsiThmAII}, nous avons démontré un théorème A de Quillen
\oo-catégorique pour les « tranches au-dessous ». On montre comment on peut
en déduire, en utilisant le théorème précédent, un théorème A pour les
« tranches au-dessus », théorème qui ne peut s'obtenir par simple
dualisation de la preuve. De même, dans~\cite{AraThmB}, le premier auteur a
démontré un théorème B de Quillen \oo-catégorique pour les « tranches
au-dessous » dont on peut déduire un théorème B pour les « tranches
au-dessus ».

\begin{organisation}
  Dans la première section, on présente les trois foncteurs nerf que l'on
  comparera dans cet article : le nerf de Street~$\NS$, le nerf
  cellulaire~$\NC$ et le nerf $n$-simplicial~$\Nsn{n}$. La catégorie
  $\Theta$ de Joyal \cite{JoyalTheta}, nécessaire à la définition du nerf
  cellulaire, est introduite en utilisant la description par produit en
  couronne due à Berger~\cite{BergerWreath}.

  Dans la deuxième section, on associe à deux \oo-catégories strictes~$A$
  et~$B$ une \oo-catégorie stricte $A \otimest B$, quotient du produit de
  Gray $A \otimes B$. On montre que cette opération admet des $\Hom$
  internes $\Homidt$ et $\Homigt$, variantes des $\Hom$ internes du produit
  de Gray, qui joueront un rôle important dans la section suivante. On
  étudie les compatibilités de $\Homidt$ et $\Homigt$ avec les
  transformations oplax et lax.

  Dans la troisième section, on compare le nerf de Street des $n$-catégories
  strictes et le nerf $n$-simplicial. On associe à toute \oo-catégorie $C$
  un objet simplicial $SC$ en \oo-catégories strictes. On en déduit un
  ensemble bisimplicial $\NS SC$ dont on montre que la diagonale est
  faiblement équivalente au nerf de Street $NC$. Pour ce faire, on
  interprète les $(p, q)$-simplexes de $\NS SC$ en termes de l'opération
  $\otimest$ et on utilise les résultats de la section précédente, ainsi que
  l'existence de certains rétractes par transformation qu'on établira
  dans un appendice. En itérant ce résultat, on obtient la comparaison de
  nerfs recherchée.

  La quatrième section est consacrée à des généralités sur des techniques de
  comparaison de nerfs basée sur la théorie des structures d'asphéricité de
  Grothendieck~\cite{GrothPS}. On commence par rappeler les bases de la
  théorie de l'homotopie de Grothendieck et préciser ce qu'on entend par
  comparer des foncteurs nerf généraux. On énonce deux résultats, le premier
  dû à Grothendieck et le second provenant de ses idées, donnant des
  conditions suffisantes pour comparer de tels foncteurs nerf.

  Dans la cinquième section, on compare le nerf de Street et le nerf
  cellulaire. On commence par montrer que le foncteur canonique $\cDelta^n
  \to \Theta_n$, où $\Theta_n$ désigne la version $n$-dimensionnelle de la
  catégorie $\Theta$ de Joyal, est asphérique. On en déduit, par un
  résultat de Grothendieck, une comparaison du nerf $n$-simplicial et du
  nerf $n$-cellulaire, et donc, en vertu des résultats de la troisième
  section, du nerf de Street $n$-catégorique et du nerf $n$-simplicial. On
  obtient alors la comparaison dans le cas \oo-catégorique en utilisant un
  des résultats généraux exposés dans la section précédente.

  La sixième section est consacrée aux applications. On définit les
  équivalences de Thomason \oo-catégoriques \forlang{via} le nerf de Street.
  En utilisant le résultat bisimplicial de la troisième section, on montre
  que les équivalences de type Dwyer-Kan pour les équivalences de Thomason
  sont des équivalences de Thomason. On montre par ailleurs, en utilisant la
  comparaison du nerf de Street et du nerf cellulaire, que les équivalences
  de Thomason sont stables par les dualités de la catégorie des
  \oo-catégories strictes. On en déduit que les théorèmes A et B de Quillen
  \oo-catégoriques pour les «~tranches au-dessous~» que nous avons établis
  dans de précédents travaux entraînent des résultats analogues pour les
  «~tranches au-dessus~».

  Enfin, on termine par un appendice dont le but est de montrer l'existence
  de deux rétractions de rétracte par transformation lax ou oplax, notion
  dont on rappelle la définition, qui sont au cœur de la preuve
  bisimpliciale de la troisième section.  La première de ces rétractions,
  $\On{q} \to \On{0}$, a été introduite dans nos précédents travaux ; quant
  à la seconde, $\On{p} \to \Deltan{p}$, on construit la transformation
  oplax qui lui est associée en utilisant le théorie des complexes dirigés
  augmentés de Steiner \cite{Steiner}.
\end{organisation}

\begin{notations}
  On notera $\ooCat$ la catégorie des \oo-catégories strictes
  et des \oo-foncteurs stricts entre celles-ci. Toutes les
  \oo-catégories et tous les \oo-foncteurs apparaissant dans ce
  texte seront stricts et on se permettra donc d'omettre cet adjectif. Pour
  tout $n \ge 0$, on notera $\Dn{n}$ le $n$\=/disque, objet de $\ooCat$
  représentant le foncteur associant à une \oo-catégorie
  l'ensemble de ses $n$-cellules.
  \[
    \shorthandoff{;:}
    \Dn{0} = \bullet
    \quad\quad
    \Dn{1} = \xymatrix{\bullet \ar[r] & \bullet}
    \quad\quad
    \Dn{2} = \xymatrix@C=3pc@R=3pc{\bullet \ar@/^2.5ex/[r]_{}="0"
      \ar@/_2.5ex/[r]_{}="1"
      \ar@2"0";"1"
    &  \bullet}
    \qquad
    \Dn{3} = \xymatrix@C=3pc@R=3pc{\bullet \ar@/^3ex/[r]_(.47){}="0"^(.53){}="10"
      \ar@/_3ex/[r]_(.47){}="1"^(.53){}="11"
      \ar@<2ex>@2"0";"1"_{}="2" \ar@<-2ex>@2"10";"11"^{}="3"
      \ar@3"3";"2"_{}
    &  \bullet
    }
  \]
  Si $C$ est une \oo-catégorie et $c$ et $c'$ sont deux objets de $C$, on
  notera $\Homi_C(c, c')$ la \oo-catégorie des $j$-cellules de $C$, pour $j
  > 0$, de $0$-source $c$ et de $0$-but $c'$. Plus généralement, pour~$i \ge
  0$, si $x$ et $y$ sont deux $i$-cellules parallèles de $C$, on
  notera~$\Homi_C(x, y)$ la \oo-catégorie des $j$-cellules de $C$, pour $j >
  i$, de $i$-source $x$ et de $i$-but $y$. Pour $i \ge 0$, l'ensemble des
  $i$-cellules de $C$ sera noté $C_i$. On notera parfois $\Ob(C)$ pour
  $C_0$.

  La catégorie des ensembles sera notée $\Ens$ et celle des petites
  catégories $\Cat$. Pour~$n \ge 0$, on notera $\nCat{n}$ la catégorie des
  $n$-catégories (strictes). On considérera souvent ces catégories comme des
  sous-catégories pleines de $\ooCat$. Le foncteur d'inclusion $\Ens \hookto
  \ooCat$ admet pour adjoint à droite le foncteur $\Ob : \ooCat \to \Ens$
  qui associe à une \oo-catégorie l'ensemble de ses objets, et admet un adjoint à
  gauche, qu'on notera $\pi_0 : \ooCat \to \Ens$, qui associe à une
  \oo-catégorie l'ensemble de ses composantes connexes.
\end{notations}

\section{Trois foncteurs nerf}

Dans cette section, nous allons rappeler les définitions de trois foncteurs
nerf classiques pour les $n$-catégories : le nerf de Street, le nerf
cellulaire et le nerf multi-simplicial.

\subsection{Le nerf de Street}

\begin{paragraph}
On notera $\cDelta$ la catégorie des simplexes. Rappelons que ses objets
sont les ensembles ordonnés
\[
  \Deltan{n} = \{0 < 1 < \cdots < n\},
\]
pour $n \ge 0$, et que ses morphismes sont les applications croissantes au
sens large entre ceux-ci. On considérera parfois $\cDelta$ comme une
sous-catégorie pleine de $\Cat$, et donc de~$\ooCat$.

La \ndef{catégorie des ensembles simpliciaux} est la catégorie
$\pref{\cDelta}$ des préfaisceaux sur $\cDelta$.
\end{paragraph}

\begin{paragraph}\label{paragr:def_orient_inform}
Dans~\cite{StreetOrient}, Street construit un foncteur $\cO : \cDelta \to
\ooCat$, appelé \ndef{objet cosimplicial des orientaux}. Pour $n \ge 0$, la
\oo-catégorie~$\cO(\Deltan{n})$, qui se trouve être une $n$-catégorie, est
appelée le \ndef{$n$-ième oriental} et est notée~$\On{n}$.
On renvoie le lecteur à \cite[paragraphe 3.11]{AraMaltsiThmAI}, par exemple,
pour une définition de l'objet cosimplicial des orientaux utilisant la
théorie des complexes dirigés augmentés de Steiner~\cite{Steiner}. Voici une
représentation graphique des premiers orientaux :
  \[
    \shorthandoff{;}
    \On{0} = \Dn{0} = \xymatrix{\{0\}}, \qquad
    \On{1} = \Dn{1} = \xymatrix{0 \ar[r] & 1},
    \qquad
    \On{2} =
    \raisebox{1.5pc}{
    $\xymatrix{
      & 2
      \\
      0 \ar[r] \ar[ur]_{}="s" & 1 \ar[u]
      \ar@{}"s";[]_(0.05){}="ss"_(0.85){}="tt"
      \ar@2"ss";"tt"
    }$
    }
    \text{,}
  \]
  \[
    \shorthandoff{;}
    \On{3} =
    \raisebox{1.5pc}{
    $\xymatrix{
      0 \ar[r]_(0.60){}="03" \ar[d] \ar[dr]_{}="02"_(0.60){}="02'" &
      3
      &
      0 \ar[r]_(0.40){}="03'" \ar[d]_{}="t3"
        &
      3
      \\
      1 \ar[r] & 2 \ar[u]_{}="s3"
      &
      1 \ar[r] \ar[ur]_{}="13"_(0.40){}="13'" & 2 \ar[u]
      \ar@{}"s3";"t3"_(0.20){}="ss3"_(0.80){}="tt3"
      \ar@3"ss3";"tt3"
      \ar@{}"13";[]_(0.05){}="s123"_(0.85){}="t123"
      \ar@2"s123";"t123"
      \ar@{}"02";[lll]_(0.05){}="s012"_(0.85){}="t012"
      \ar@2"s012";"t012"
      \ar@{}"03'";"13'"_(0.05){}="s013"_(0.85){}="t013"
      \ar@2"s013";"t013"
      \ar@{}"03";"02'"_(0.05){}="s023"_(0.85){}="t023"
      \ar@2"s023";"t023"
      \\
    }$
    }
    \text{.}
  \]
\end{paragraph}

\begin{paragraph}
  Le nerf de Street est le foncteur nerf
  \[ \NS : \ooCat \to \pref{\cDelta} \]
  associé à l'objet cosimplicial des orientaux. Explicitement, si $C$ est
  une \oo-catégorie, on a, pour tout $n \ge 0$,
  \[ \NS(C)_n = \Hom_{\ooCat}(\On{n}, C). \]
\end{paragraph}

\subsection{Le nerf cellulaire}\ %

\medskip

Nous allons commencer par introduire la catégorie $\Theta$ de Joyal en
utilisant la définition par produit en couronne introduite par Berger dans
\cite{BergerWreath}.

\begin{paragraph}
  Soit $C$ une catégorie. On définit une catégorie $\wrDelta{C}$ de la
  manière suivante.
  \begin{itemize}
    \item Les objets de $\cDelta \wr C$ sont les $(\Deltan{n}; c_1, \dots,
      c_n)$, où $n \ge 0$ et $c_1, \dots, c_n$ sont $n$ objets de~$C$.
    \item Si $(\Deltan{n}; c_1, \dots, c_n)$ et $(\Deltan{n'}; c'_1, \dots,
      c'_{n'})$ sont deux objets de $\wrDelta{C}$, les morphismes du premier
      vers le second sont les $(\phi; (f_{i',i}))$, où $\phi :
      \Deltan{n} \to \Deltan{n'}$ est un morphisme de~$\cDelta$ et les
      $f_{i',i} : c^{}_i \to c'_{i'}$, pour $1 \le i \le n$ et $\phi(i - 1) <
      i' \le \phi(i)$, sont des morphismes de~$C$.
    \item Si
       \[
         (\Deltan{n}; c_1, \dots, c_n)
         \xto{(\phi; (f_{i',i}))}
         (\Deltan{n'}; c'_1, \dots, c'_{n'})
         \xto{(\phi'; (f'_{i'',i'}))}
         (\Deltan{n''}; c''_1, \dots, c''_{n''})
       \]
       sont deux morphismes composables, leur composé est le morphisme
       \[
         (\Deltan{n}; c_1, \dots, c_n)
         \xto{(\phi''; (f''_{i'',i}))}
         (\Deltan{n''}; c''_1, \dots, c''_{n''})
       \]
       où 
       \[ \phi'' = \phi'\phi \quadet f''_{i'',i} = f'_{i'',i'}f^{}_{i',i}, \]
       $i'$ étant l'unique entier tel que $\phi(i-1) < i' \le \phi(i)$ et
       $\phi'(i'-1) < i'' \le \phi'(i')$.
     \item Si $(\Deltan{n};c_1, \dots, c_n)$ est un objet de $\wrDelta{C}$,
       l'identité de cet objet est le morphisme $(\id{\Deltan{n}};
       (\id{c_i})_{1 \le i = i' \le n})$.
  \end{itemize}
  On vérifie qu'on obtient bien ainsi une catégorie. Notons que si $e$
  désigne la catégorie ponctuelle, la catégorie $\wrDelta{e}$ est
  canoniquement isomorphe à $\cDelta$.

  La construction $C \mapsto \wrDelta{C}$ s'étend en un
  foncteur de manière évidente : en particulier, si $u : C \to D$ est un
  foncteur, on obtient un foncteur $\wrDelta{u} : \wrDelta{C} \to
  \wrDelta{D}$.
\end{paragraph}

\begin{paragraph}\label{paragr:proj_wr}
  Si $C$ est une catégorie, on dispose de foncteurs
  \[ \Sigma : C \to \wrDelta{C} \quadet \pi : \wrDelta{C} \to \cDelta. \]
  Le premier de ces foncteurs est défini, sur les objets et les
  morphismes, par
  \[
    \begin{split}
      c & \mapsto (\Deltan{1}; c) \\
      f & \mapsto (\id{\Deltan{1}}; (f)_{i'=i=1})
    \end{split}
  \]
  et le second par
  \[
    \begin{split}
      (\Deltan{n}; c_1, \dots, c_n) & \mapsto \Deltan{n} \\
      (\phi, (f_{i',i})) & \mapsto \phi \bpbox{.}
    \end{split}
  \]
  Notons que ce second foncteur n'est autre que le foncteur $\wrDelta{p} :
  \wrDelta{C} \to \cDelta$, où $p : C \to e$ désigne l'unique foncteur
  de $C$ vers la catégorie ponctuelle $e$ et où on a identifié $\wrDelta{e}$
  et $\cDelta$.
\end{paragraph}

\begin{paragraph}
  Pour $n \ge 0$, on définit la catégorie $\Theta_n$ de Joyal par récurrence
  de la manière suivante :
  \begin{itemize}
    \item $\Theta_0$ est la catégorie ponctuelle $\ponct$ ;
    \item $\Theta_{n+1} = \wrDelta{\Theta_n}$.
  \end{itemize}
  En particulier, $\Theta_1$ est la catégorie
  $\wrDelta{\ponct}$ qu'on identifiera toujours à la
  catégorie $\cDelta$. On a donc $\Theta_1 = \cDelta$.

  On définit, toujours par récurrence sur $n$, un foncteur $i_n :
  \Theta_n \to \Theta_{n+1}$:
  \begin{itemize}
    \item le foncteur $i_0 : \Theta_0 = \ponct \to \Theta_1 = \cDelta$
      correspond à l'objet $\Deltan{0}$ de $\cDelta$ ;
    \item $i_{n+1} = \wrDelta{i_n}$.
  \end{itemize}
  On vérifie facilement que le foncteur $i_n : \Theta_n \to \Theta_{n+1}$
  est pleinement fidèle, et on le considérera comme un foncteur d'inclusion.

  La catégorie $\Theta$ de Joyal est définie par
  \[ \Theta = \limind \Theta_n. \]
  On dispose de foncteurs d'inclusion $\Theta_n \hookto \Theta$.

  Un \ndef{ensemble cellulaire} est un préfaisceau sur $\Theta$ et un
  \ndef{ensemble $n$-cellulaire} un préfaisceau sur $\Theta_n$.
\end{paragraph}

\begin{paragraph}
  Soit $\C$ une catégorie admettant des produits finis et un objet initial.
  On définit un foncteur $W_\C : \wrDelta{\C} \to \CCat{\C}$, où
  $\CCat{\C}$ désigne la catégorie des catégories enrichies dans $\C$, de la
  manière suivante.  Soit $S = (\Deltan{n}; C_1, \dots, C_n)$ un objet de
  $\wrDelta{\C}$.  On commence par définir un graphe~$G_S$ enrichi dans $\C$:
  \begin{itemize}
    \item les objets de $G_S$ sont les entiers $0, \dots, n$;
    \item si $i$ et $j$ sont deux objets, l'objet des flèches de
      $i$ vers $j$ est $C_j$ si $j = i + 1$ et l'objet initial $\emptyset$ de $\C$
      sinon.
  \end{itemize}
  Par définition, $W_\C(S)$ est la catégorie enrichie dans $\C$ librement
  engendrée par ce graphe $G_S$ enrichi dans $\C$. Explicitement, si $i$ et
  $j$ sont deux objets, on a
  \[
    \Homi_{W_\C(S)}(i, j) =
        \begin{cases}
          \prod_{i < k \le j} C_k & \text{si $i \le j$,} \\
          \emptyset & \text{sinon.}
        \end{cases}
  \]
  On vérifie que le foncteur $W_\C : \wrDelta{\C} \to \CCat{\C}$ est pleinement
  fidèle (voir~\cite[proposition~3.5]{BergerWreath}).
\end{paragraph}

\begin{paragraph}\label{paragr:constr_W}
  Dans le cas où $\C = \ooCat$, la construction du paragraphe précédent
  fournit un foncteur $W : \wrDelta{\ooCat} \to \ooCat$, la catégorie
  $\CCat{(\ooCat)}$ s'identifiant à $\ooCat$.

  Si $n \ge 0$ est un entier et $C_1, \dots, C_n$ sont des \oo-catégories, on notera
  parfois $\Deltan{n} \wr (C_1, \dots, C_n)$ la \oo-catégorie
  $W((\Deltan{n}; C_1, \dots, C_n))$. Dans le cas où tous les $C_i$ sont
  égaux à une même \oo-catégorie $C$, on notera plus simplement~$\Deltan{n}
  \wr C$ cette même \oo-catégorie.

  On vérifie facilement que la \oo-catégorie $\Deltan{n} \wr (C_1, \dots, C_n)$
  est la limite inductive du diagramme
  \[
    \xymatrix@C=1pc{
      \Sigma'C_n & & \Sigma'C_{n-1} & & \cdots & & \Sigma'C_1 \\
                   & \Dn{0} \ar[lu]^{0} \ar[ru]_{1} & & \Dn{0} \ar[lu]^{0}
                   & \cdots & \Dn{0} \ar[ur]_1 & \bpbox{,}\\
    }
  \]
  où, si $D$ est une \oo-catégorie, $\Sigma'D$ désigne la \oo-catégorie
  dont les objets sont $0$ et~$1$, et telle que
  \[
    \Homi_{\Sigma'D}(\e, \e')
    =
    \begin{cases}
      D & \text{si $\e = 0$ et $\e' = 1$,} \\
      \Dn{0} & \text{si $\e = \e'$,} \\
      \emptyset & \text{si $\e = 1$ et $\e' = 0$},
    \end{cases}
  \]
  la composition étant définie de la manière évidente.

  Il est par ailleurs immédiat que le foncteur $W$ se restreint pour tout $n
  \ge 0$ en un foncteur $W_n : \wrDelta{\nCat{n}} \to \nCat{(n+1)}$, qui
  n'est autre que le foncteur $W_{\nCat{n}}$ du paragraphe précédent.
\end{paragraph}

\begin{paragraph}\label{paragr:def_R_n}
  On définit un foncteur $R_n : \Theta_n \to \nCat{n}$ par récurrence sur $n
  \ge 0$:
  \begin{itemize}
    \item le foncteur $R_0 : \Theta_0 = \ponct \to \nCat{0} = \Ens$
      correspond au singleton ;
    \item le foncteur $R_{n+1}$ est le composé
      \[
        \Theta_{n+1} = \wrDelta{\Theta_n} \xto{\wrDelta{R_n}}
        \wrDelta{\nCat{n}} \xto{W_n} \nCat{(n+1)} \pbox{.}
      \]
  \end{itemize}
  On obtient un foncteur $R : \Theta \to \ooCat$ en passant à la limite
  inductive.

  Par exemple, si on considère l'objet $S = (\Deltan{3};\Deltan{3},
  \Deltan{0}, \Deltan{2})$ de $\Theta_2$, on a
  {
    \shorthandoff{;}
  \[ R_2(S) =
     \xymatrix{
       \bullet
       \ar@/^5ex/[r]^{}="a"
       \ar@/^1.7ex/[r]^{}="b"
       \ar@/_1.7ex/[r]^{}="c"
       \ar@/_5ex/[r]_{}="d"
       \ar@2"a";"b"
       \ar@2"b";"c"
       \ar@2"c";"d"
       &
       \bullet
       \ar[r]^{}="e"
       &
       \bullet
       \ar@/^4ex/[r]^{}="f"
       \ar[r]^{}="g"
       \ar@/_4ex/[r]^{}="h"
       \ar@2"f";"g"
       \ar@2"g";"h"
       &
       \bullet
       \pbox{.}
     }
  \]
  }

  Il résulte de la pleine fidélité du foncteur $W_n$ que le foncteur $R :
  \Theta \to \ooCat$ est également pleinement fidèle. Ainsi, on considérera
  souvent $\Theta$ comme une sous-catégorie pleine de~$\ooCat$.
\end{paragraph}

\begin{paragraph}
  Le \ndef{nerf cellulaire} 
  \[ \NC : \ooCat \to \pref{\Theta} \]
  est le foncteur
  nerf associé au foncteur $R : \Theta \hookto \ooCat$ du paragraphe
  précédent. Explicitement, si $C$ est une \oo-catégorie et $T$ est un objet
  de~$\Theta$, on a
  \[ \NC(C)_T = \Hom_{\ooCat}(R(T), C). \]
  De même, pour tout $n \ge 0$, le foncteur $R_n : \Theta_n \hookto
  \nCat{n}$ induit un foncteur
  \ndef{nerf $n$-cellulaire} 
  \[ \NCn{n} : \nCat{n} \to \pref{\Theta_n}. \]
  Notons que ce foncteur n'est autre que le composé
  \[
    \nCat{n} \hookto \ooCat \xto{\NC} \pref{\Theta} \xto{i_n^\ast}
    \pref{\Theta_n} \pbox{,}
  \]
  où $i_n^\ast$ désigne le foncteur de restriction le long de l'inclusion $i_n :
  \Theta_n \hookto \Theta$.
\end{paragraph}

\subsection{Le nerf $n$-simplicial}

\begin{paragraph}\label{paragr:nerf_multi}
  Si $C$ est une catégorie, on définit un foncteur $\mu : \cDelta \times C
  \to \wrDelta{C}$, sur les objets et les morphismes, par
  \[
    \begin{split}
      (\Deltan{n}, c) & \mapsto (\Deltan{n}; c, \dots, c) \\
      (\phi, f) & \mapsto (\phi; (f)_{i',i}) \bpbox{.}
    \end{split}
  \]

  On en déduit par récurrence sur $n \ge 0$ un foncteur $m_n : \cDelta^n \to
  \Theta_n$:
  \begin{itemize}
    \item $m_0$ est l'identité de $\cDelta^0 = \ponct = \Theta_0$;
    \item $m_{n+1}$ est le composé
      \[
        \cDelta^{n+1} = \cDelta \times \cDelta^{n} \xto{\cDelta \times m_n} \cDelta \times
        \Theta_n \xto{\mu} \wrDelta{\Theta_n} = \Theta_{n+1} \pbox{.}
      \]
  \end{itemize}

  Pour $n \ge 0$, on obtient ainsi un foncteur $M_n : \cDelta^n \to
  \nCat{n}$ en composant
  \[ \cDelta^n \xto{m_n} \Theta_n \hookto \nCat{n} \pbox{.} \]
  Explicitement, avec les notations du paragraphe~\ref{paragr:constr_W},
  on a
  \[ M_n(\Deltan{i_1}, \dots, \Deltan{i_n})
    = \Deltan{i_1} \wr \cdots \wr \Deltan{i_n},
   \]
   où l'opération $\wr$ est implicitement parenthésée à droite.

  Par exemple, on a
  {
    \shorthandoff{;}
  \[
    m_2(\Deltan{2}, \Deltan{3}) = (\Deltan{2}; \Deltan{3}, \Deltan{3})
    \quadet
    M_2(\Deltan{2}, \Deltan{3}) =
     \xymatrix{
       \bullet
       \ar@/^5ex/[r]^{}="a"
       \ar@/^1.7ex/[r]^{}="b"
       \ar@/_1.7ex/[r]^{}="c"
       \ar@/_5ex/[r]^{}="d"
       \ar@2"a";"b"
       \ar@2"b";"c"
       \ar@2"c";"d"
       &
       \bullet
       \ar@/^5ex/[r]^{}="a"
       \ar@/^1.7ex/[r]^{}="b"
       \ar@/_1.7ex/[r]^{}="c"
       \ar@/_5ex/[r]^{}="d"
       \ar@2"a";"b"
       \ar@2"b";"c"
       \ar@2"c";"d"
       &
       \bullet
       \pbox{.}
     }
  \]
  }
\end{paragraph}

\begin{paragraph}
  Soit $n \ge 0$. Le \ndef{nerf $n$-simplicial}
  \[ \Nsn{n} : \nCat{n} \to \pref{\cDelta^n} \]
  est le foncteur nerf associé au foncteur $M_n :
  \cDelta^n \to \nCat{n}$ du paragraphe précédent. Explicitement, si $C$ est
  une $n$-catégorie et $i_1, \dots, i_n$ sont des entiers, on a
  \[
    \begin{split}
      \Nsn{n}(C)_{i_1, \dots, i_n}
      & = \Hom_{\nCat{n}}(M_n(\Deltan{i_1}, \dots, \Deltan{i_n}), C) \\
      & = \Hom_{\nCat{n}}(\Deltan{i_1} \wr \cdots \wr \Deltan{i_n}, C).
    \end{split}
  \]
\end{paragraph}

\section{Un quotient du produit de Gray}

\begin{paragraph}
  Si $A$ et $B$ sont deux \oo-catégories, on notera $A \otimes B$ leur
  produit de Gray. On renvoie le lecteur à l'appendice A de
  \cite{AraMaltsiJoint} pour une étude systématique du produit de Gray,
  utilisant une définition basée sur une idée de Steiner \cite{Steiner}.

  On rappelle que le produit de Gray définit sur $\ooCat$ une structure de
  catégorie monoïdale (non symétrique) bifermée, d'unité la \oo-catégorie
  finale $\Dn{0}$. Le fait que le produit de Gray soit bifermé signifie
  qu'on dispose de foncteurs
  \[ \HomOpLax, \HomLax : \ooCat^\o \times \ooCat \to \ooCat \]
  et de bijections
  \[ \Hom_{\ooCat}(A, \HomOpLax(B, C)) \simeq \Hom_{\ooCat}(A \otimes B, C)
      \simeq \Hom_{\ooCat}(B, \HomLax(A, C)),
  \]
  naturelles en $A$, $B$ et $C$ dans $\ooCat$. En particulier, le produit
  de Gray commute aux limites inductives en chacun de ses arguments.
\end{paragraph}

\begin{paragraph}
  Soient $A$ et $B$ deux \oo-catégories. Par adjonction, les $0$-cellules de
  la \oo-catégorie $\HomOpLax(A, B)$ correspondent aux \oo-foncteurs de $A$
  vers $B$. Si $u, v : A \to B$ sont deux \oo-foncteurs, une
  \ndef{transformation oplax} de $u$ vers $v$ est une $1$-cellule de
  $\HomOpLax(A, B)$ de source $u$ et de but $v$. Par adjonction, une telle
  transformation correspond à un \oo-foncteur $h : \Dn{1} \otimes A \to B$
  rendant commutatif le diagramme
  \[
    \xymatrix@C=3pc{
    A \ar[dr]^u \ar[d]_{0 \otimes A} \\
    \Dn{1} \otimes A \ar[r]^h & D \\
    A \ar[ur]_v \ar[u]^{1 \otimes A} & \pbox{,} \\
    }
  \]
  où $\Dn{1} = 0 \to 1$ et où on a identifié $A$ avec $\Dn{0} \otimes A$. Si $x$
  est une $n$-cellule de $A$, on peut lui associer par une telle transformation
  une $(n+1)$\=/cellule~$h_x$ en composant
  \[
    \Dn{n+1} \to \Dn{1} \otimes \Dn{n} \xto{\Dn{1} \otimes x} \Dn{1} \otimes A
    \xto{h} B \pbox{,}
  \]
  où la première flèche correspond à l'unique $(n+1)$-cellule non triviale
  (c'est-à-dire qui n'est pas une identité)
  de $\Dn{1} \otimes \Dn{n}$. On montre que les $h_x$, où $x$ varie parmi
  les cellules de $A$, déterminent complètement~$h$. (On renvoie, par
  exemple, à l'appendice~B de \cite{AraMaltsiJoint} pour une étude
  systématique des transformations oplax.)

  On définit de même la notion de \ndef{transformation lax} en remplaçant
  l'usage de $\HomOpLax(A, B)$ par celui de $\HomLax(A, B)$. Ainsi, une
  transformation lax entre \oo-foncteurs de $A$ vers $B$ est donnée par un
  \oo-foncteur $A \otimes \Dn{1} \to B$.
\end{paragraph}

\begin{paragraph}\label{paragr:def_otimest}
  Soient $A$ et $B$ deux \oo-catégories. On définit une \oo-catégorie $A
  \otimest B$ par le carré cocartésien
  \[
    \xymatrix@C=3.5pc@R=2.5pc{
      A \otimes \Ob(B) \ar[r]^-{p \otimes \Ob(B)} \ar[d]_{A \otimes i} &
      \pi_0(A) \otimes \Ob(B) \ar[d] \\ A \otimes B \ar[r] & \pushoutcorner A
      \otimest B \bpbox{,}
    }
  \]
  où $\Ob(B)$ désigne l'ensemble des objets de $B$ vu comme une \oo-catégorie
  discrète, $i : \Ob(B) \hookto B$ l'inclusion canonique, $\pi_0(A)$
  l'ensemble des composantes connexes de $A$ vu comme une \oo-catégorie
  discrète et $p : A \to \pi_0(A)$ la projection canonique. (Notons que le
  produit de Gray de deux ensembles n'est autre que leur produit
  cartésien.) La \oo-catégorie $A \otimest B$ est ainsi munie d'un
  \oo-foncteur canonique $A \otimes B \to A \otimest B$, qui est un
  épimorphisme car image directe d'un épimorphisme.

  Cette construction est clairement fonctorielle en $A$ et $B$, et on
  obtient donc un foncteur
  \[ \otimest : \ooCat \times \ooCat \to \ooCat, \]
  ainsi qu'une transformation naturelle canonique $A \otimes B \to A
  \otimest B$. On se gardera de croire que $\otimest$ définit une produit
  monoïdal sur $\ooCat$. (On peut par exemple montrer que $(\Dn{1} \otimest
  \Dn{1}) \otimest \Dn{1}$ n'est pas isomorphe à $\Dn{1} \otimest (\Dn{1}
  \otimest \Dn{1})$.)

  On vérifie immédiatement qu'on a des isomorphismes canoniques
  \[ \Dn{0} \otimest A \simeq A \quadet A \otimest \Dn{0} \simeq \pi_0(A),
  \]
  naturels en $A$ dans $\ooCat$.
\end{paragraph}

\begin{proposition}
  Soient $A$ et $B$ deux \oo-catégories. On a des bijections canoniques
  \[ \Ob(A \otimest B) \simeq \pi_0(A) \times \Ob(B)
  \quadet
     \pi_0(A \otimest B) \simeq \pi_0(A) \times \pi_0(B), \]
  naturelles en $A$ et $B$.
\end{proposition}

\begin{proof}
  Considérons le carré cocartésien
  \[
    \xymatrix{
      A \otimes \Ob(B) \ar[r]
         \ar[d]
         &
      \pi_0(A) \otimes \Ob(B) \ar[d] \\ A \otimes B \ar[r] & \pushoutcorner A
      \otimest B \bpbox{.}
    }
  \]
  En appliquant à ce carré le foncteur $\Ob : \ooCat \to \Ens$, qui commute aux limites
  inductives, on obtient un nouveau carré cocartésien. Il résulte de
  l'isomorphisme naturel $\Ob(C \otimes D) \simeq \Ob(C) \times \Ob(D)$ que
  la flèche verticale de gauche de ce second carré est un isomorphisme. On
  en déduit
  qu'il en est de même de celle de droite et on a donc
  \[
    \Ob(A \otimest B) \simeq \Ob(\pi_0(A) \otimes \Ob(B))
    \simeq \pi_0(A) \times \Ob(B).
  \]
  Appliquons maintenant au carré cocartésien considéré au début de cette
  preuve le foncteur $\pi_0 : \ooCat \to \Ens$, qui commute également aux
  limites inductives. Il résulte de l'isomorphisme naturel $\pi_0(C \otimes
  D) \simeq \pi_0(C) \times \pi_0(D)$ (voir par exemple le cas $n = 0$
  de \cite[Proposition A.27]{AraMaltsiJoint}) que la flèche horizontale du
  haut du carré ainsi obtenu est un isomorphisme. Il en est donc de même de
  la flèche du bas et on a donc
  \[
    \pi_0(A \otimest B) \simeq \pi_0(A \otimes B) \simeq \pi_0(A) \times
    \pi_0(B),
  \]
  ce qu'il fallait démontrer.
\end{proof}

\begin{proposition}\label{prop:iso_otimest}
  Soient $A$, $B$ et $C$ trois \oo-catégories. On a un isomorphisme
  canonique
  \[ A \otimest (B \otimest C) \simeq (A \otimes B) \otimest C, \]
  naturel en $A$, $B$ et $C$.
\end{proposition}

\begin{proof}
  En utilisant le fait que le foncteur $A \otimes {-}$ commute aux limites
  inductives et la proposition précédente, on obtient
  \[
    \begin{split}
      A \otimest (B \otimest C)
      & =
      \big(A \otimes (B \otimest C)\big) \amalg_{A \otimes \Ob(B \otimest C)}
      \big(\pi_0(A) \otimes \Ob(B \otimest C)\big) \\
      & \simeq
      \big((A \otimes B \otimes C) \amalg_{A \otimes B \otimes \Ob(C)} (A \otimes \pi_0(B)
      \otimes \Ob(C))\big) \\
      & \phantom{\simeq1} \qquad\qquad \amalg_{A \otimes \pi_0(B) \otimes \Ob(C)}
      \big(\pi_0(A) \otimes \pi_0(B) \otimes \Ob(C)\big) \\
      & \simeq
      (A \otimes B \otimes C) \amalg_{A \otimes B \otimes \Ob(C)}
      (\pi_0(A) \otimes \pi_0(B) \otimes \Ob(C)) \\
      & \simeq
      \big((A \otimes B) \otimes C\big) \amalg_{(A \otimes B) \otimes \Ob(C)}
      (\pi_0(A \otimes B) \otimes \Ob(C)) \\
      & =
      (A \otimes B) \otimest C,
    \end{split}
  \]
  d'où le résultat.
\end{proof}

\begin{remark}
  Ainsi, on dispose d'un \oo-foncteur canonique
  \[ A \otimest (B \otimest C) \simeq (A \otimes B) \otimest C \to (A
  \otimest B) \otimest C \pbox{,} \]
  et on peut montrer que le produit $\otimest$ munit $\ooCat$ d'une
  structure de catégorie monoïdale lax en un sens adéquat.
\end{remark}

\begin{proposition}
  Le foncteur $\otimest : \ooCat \times \ooCat \to \ooCat$ commute aux
  limites inductives en chacune de ses variables.
\end{proposition}

\begin{proof}
  Cela résulte du fait que le foncteur $\otimes$ a cette propriété, que les
  foncteurs $A \mapsto \pi_0(A)$ et $B \mapsto \Ob(B)$ commutent tous les
  deux aux limites inductives et que, enfin, les limites inductives
  commutent aux sommes amalgamées.
\end{proof}
 
\begin{proposition}
  Il existe des foncteurs
  \[ \Homidt, \Homigt : \ooCat^\o \times \ooCat \to \ooCat \]
  et des bijections
  \[
    \Hom_{\ooCat}(A, \Homidt(B, C))
    \simeq
    \Hom_{\ooCat}(A \otimest B, C)
    \simeq
    \Hom_{\ooCat}(B, \Homigt(A, C))
  \]
  naturelles en $A$, $B$ et $C$ dans $\ooCat$.
\end{proposition}

\begin{proof}
  La catégorie $\ooCat$ étant localement présentable, cela résulte
  formellement du fait que le foncteur $\otimest$ commute aux limites
  inductives en chacune de ses variables.
\end{proof}

\begin{paragraph}\label{paragr:desc_Homit}
  Par adjonction, le \oo-foncteur canonique $C \otimes D \to C \otimest D$
  induit des \oo-foncteurs
  \[
     \Homidt(A, B) \to \HomOpLax(A, B)
     \quadet
     \Homigt(A, B) \to \HomLax(A, B).
  \]
  Il résulte du fait que $A \otimes B \to A \otimest B$ est un épimorphisme
  que ces \oo-foncteurs sont des monomorphismes, et on considérera souvent
  ces \oo-foncteurs comme des inclusions.

  Explicitement, cela revient à identifier une $i$-cellule de $\Homidt(A, B)$
  avec un \oo-foncteur $\Dn{i} \otimes A \to B$ qui se factorise par
  l'épimorphisme $\Dn{i} \otimes A \to \Dn{i} \otimest A$, et une
  $i$-cellule de $\Homigt(A, B)$
  avec un \oo-foncteur de $A \otimes \Dn{i} \to B$ qui se factorise par
  l'épimorphisme $A \otimes \Dn{i} \to A \otimest \Dn{i}$.

  Puisque le \oo-foncteur canonique $\Dn{0} \otimes A \to \Dn{0} \otimest A$
  est un isomorphisme, on a
  \[
    \Ob(\Homidt(A, B)) = \Ob(\HomOpLax(A, B)) = \Hom_{\ooCat}(A, B).
  \]
  Si $i \ge 1$, en revenant à la définition de $\Dn{i} \otimest A$, on
  obtient qu'une $i$-cellule $h : \Dn{i} \otimes A \to B$ de $\HomOpLax(A,
  B)$ appartient à $\Homidt(A, B)$ si et seulement si, pour tout objet $a$
  de $A$, le composé
  \[
    \Dn{i} \otimes \Dn{0} \xto{\Dn{i} \otimes a} \Dn{i} \otimes A \xto{h} B
  \]
  se factorise par $\pi_0(\Dn{i}) \simeq \Dn{0}$. Lorsque $i = 1$,
  cette condition signifie exactement que, pour tout objet $a$ de $A$, la
  $1$-flèche $h_a$ de $B$ est une identité. On dira qu'une transformation
  oplax $h$ qui vérifie cette condition est \ndef{triviale sur les objets}.

  De même, par définition, un objet de $\HomLax(A, B)$, \oo-foncteur de $A$
  vers $B$, est dans $\Homigt(A, B)$ s'il se factorise par
  $\pi_0(A)$, c'est-à-dire s'il est localement constant. Si $i \ge 1$, une
  $i$-cellule $k : A \otimes \Dn{i} \to B$ de $\HomLax(A, B)$ est dans
  $\Homigt(A, B)$ si, pour $\e = 0, 1$, le composé
  \[
    A \otimes \Dn{0} \xto{A \otimes \e} A \otimes \Dn{i} \xto{k} B
  \]
  se factorise par $\pi_0(A)$. Or, ce composé n'est autre que la source ou
  le but, suivant que $\e$ vaille $0$ ou $1$, de la $i$-cellule $k$ de
  $\HomLax(A, B)$. Ainsi, $\Homigt(A, B)$ est la sous-\oo-catégorie pleine
  de $\HomLax(A, B)$ dont les objets sont les \oo-foncteurs localement
  constants de $A$ vers $B$.
\end{paragraph}

\begin{remark}
  Dans le cadre $2$-catégorique, les transformations oplax triviales sur les
  objets sont parfois appelées \ndef{icons} pour « Identity Component
  Oplax Natural transformations » \cite{LackIcons}.
\end{remark}

\begin{proposition}
  Soient $A$, $B$ et $C$ des \oo-catégories. On a des isomorphismes
  canoniques
    \[ \Homidt(A \otimest B, C) \simeq \HomOpLax(A, \Homidt(B, C)) \]
  et
    \[ \Homigt(A \otimes B, C) \simeq \Homigt(B, \Homigt(A, C)), \]
  naturels en $A$, $B$ et $C$.
\end{proposition}

\begin{proof}
  Soit $T$ une \oo-catégorie. Par adjonction, on a
  \[
    \Hom_{\ooCat}(T, \Homidt(A \otimest B, C)) \simeq
    \Hom_{\ooCat}(T \otimest (A \otimest B), C)
  \]
  et
  \[
    \begin{split}
    \Hom_{\ooCat}(T, \HomOpLax(A, \Homidt(B, C)))
    & \simeq
    \Hom_{\ooCat}(T \otimes A, \Homidt(B, C)) \\
    & \simeq
    \Hom_{\ooCat}((T \otimes A) \otimest B, C) \\
    \end{split}
  \]
  et on obtient le premier isomorphisme en vertu du lemme de Yoneda et de
  l'isomorphisme
  \[ T \otimest (A \otimest B) \simeq (T \otimes A) \otimest B \]
  de la proposition~\ref{prop:iso_otimest}. Le deuxième isomorphisme
  s'obtient de manière similaire en observant que
  \[ \Hom_{\ooCat}(T, \Homigt(A \otimes B, C)) \simeq \Hom_{\ooCat}((A
  \otimes B) \otimest T, C) \]
  et
  \[ \Hom_{\ooCat}(T, \Homigt(B, \Homigt(A, C)))
    \simeq
    \Hom_{\ooCat}(
    A \otimest (B \otimest T), C). \qedhere
  \]
\end{proof}

\begin{paragraph}
  Soient $A$ et $B$ deux \oo-catégories. Pour toute \oo-catégorie $S$, tout
  \oo-foncteur $h : S \otimest A \to B$ et toute \oo-catégorie~$C$,
  le composé
    \[ \Homidt(B, C) \to \Homidt(S \otimest A, C) \simeq \HomOpLax(S,
        \Homidt(A, C)) \pbox{,} \]
    où la première flèche est induite par $h$ et la seconde est un des
    isomorphismes de la proposition précédente, fournit par adjonction 
    un \oo-foncteur
    \[ \Homidt(B, C) \otimes S \to \Homidt(A, C), \]
    naturel en $A$, $B$, $S$, $h$ et $C$.

    De même, pour toute \oo-catégorie $S$, tout \oo-foncteur $k : A
    \otimes S \to B$ et toute \oo-catégorie $C$, le composé
      \[ \Homigt(B, C) \to \Homigt(A \otimes S, C) \simeq \Homigt(S,
      \Homigt(A, C)) \pbox{,} \]
      où la première flèche est induite par $k$, fournit par adjonction un
      \oo-foncteur
    \[ S \otimest \Homigt(B, C) \to \Homigt(A, C), \]
    naturel en $A$, $B$, $S$, $k$ et $C$.
\end{paragraph}

\begin{proposition}\label{prop:trans_lax_Homit}
  Soient $u, v : A \to B$ deux \oo-foncteurs et soit $C$ une \oo-catégorie.
  \begin{enumerate}
    \item Une transformation oplax de $u$ vers $v$ triviale sur les objets
      induit une transformation lax de $\Homidt(u, C)$ vers $\Homidt(v, C)$.
    \item Une transformation lax de $u$ vers $v$ induit une transformation
      oplax triviale sur les objets de $\Homigt(u, C)$ vers $\Homigt(v, C)$.
  \end{enumerate}
\end{proposition}

\begin{proof}
  Cela résulte du paragraphe précédent pour $S = \Dn{1}$, ainsi que de la
  description des \oo-foncteurs $\Dn{1} \otimest C \to D$ donnée au
  paragraphe~\ref{paragr:desc_Homit}.
\end{proof}

\section{Comparaison du nerf de Street et du nerf $n$-simplicial}

Dans toute cette section, on fixe un entier $n \ge 1$.

\begin{paragraph}
  On notera $\delta : \cDelta \to \cDelta^n$ le foncteur diagonal et
  $\delta^\ast : \pref{\cDelta^n} \to \pref{\cDelta}$ le foncteur qui s'en
  déduit par précomposition. Ainsi, si $X$ est un ensemble $n$-simplicial,
  alors $\delta^\ast X$ est l'ensemble simplicial défini par $(\delta^\ast
  X)_p = X_{p, \dots, p}$. On appellera \ndef{équivalences faibles $n$-simpliciales
  diagonales} les morphismes d'ensembles $n$-simpliciaux $f$ tels que
  $\delta^\ast f$ soit une équivalence d'homotopie faible simpliciale.
\end{paragraph}

On rappelle le résultat classique suivant :

\begin{lemma_bi_simpl}
  Soit $f$ un morphisme d'ensembles bisimpliciaux. Si pour tout $p \ge 0$ le
  morphisme d'ensembles simpliciaux $f_{p, \bullet}$ \resp{$f_{\bullet,
  p}$}, obtenu en fixant la première variable \resp{la deuxième variable}
  est une équivalence faible, alors $f$ est une équivalence faible
  diagonale.
\end{lemma_bi_simpl}

\begin{proof}
  Voir par exemple \cite[Chapitre XII, paragraphe 4.3]{BousKan}, ou
  \cite[proposition 2.1.7]{CisinskiLFM} pour une preuve plus moderne.
\end{proof}

Plus généralement :

\begin{lemmamulti}\label{lemme:multi_simpl}
  Soit $f$ un morphisme d'ensembles $n$\=/simpliciaux.
  S'il existe un sous-ensemble $I$ de $\{1, \dots, n\}$, de cardinal $k < n$,
  tel que, pour toute famille d'entiers positifs $(p_i)_{i \in I}$, le morphisme
  d'ensembles $(n-k)$-simpliciaux obtenu à partir de $f$ en fixant, pour tout
  $i$ dans $I$, la $i$-ième variable à la valeur $p_i$ est une équivalence
  faible diagonale, alors $f$ est une équivalence faible diagonale.
\end{lemmamulti}

\begin{proof}
  Cela résulte du lemme bisimplicial par récurrence.
\end{proof}

Le but de cette section est de produire, pour $C$ une $n$-catégorie, un
zigzag d'équivalences faibles simpliciales entre la diagonale du nerf
$n$-simplicial $\delta^\ast\Nsn{n}C$ de $C$ et le nerf de Street $NC$ de
$C$. Pour ce faire, on va introduire une construction intermédiaire $SC$
permettant de décomposer le nerf $n$-simplicial en $n$ étapes au sens où on
aura~$\Nsn{n}C \simeq S^nC$.

\begin{paragraph}\label{paragr:def_SC}
  Si $C$ est une \oo-catégorie, on notera $SC$ l'objet simplicial dans
  $\ooCat$ défini par
  \[
    \begin{split}
      \cDelta^\o & \to \ooCat \\
      \Deltan{p} & \mapsto S_pC = \Homidt(\Deltan{p}, C).
    \end{split}
  \]
\end{paragraph}

\begin{proposition}\label{prop:pu_SC}
  Soient $C$ une \oo-catégorie quelconque, $T$ une \oo-catégorie connexe et
  $p \ge 0$. On a une bijection canonique
  \[
    \Hom_{\ooCat}(T, S_pC) \simeq \Hom_{\ooCat}(\Deltan{p} \wr T, C)
  \]
  (voir le paragraphe~\ref{paragr:constr_W} pour la définition de
  $\Deltan{p} \wr T$) définissant un isomorphisme
  d'ensembles simpliciaux naturel en $C$ et $T$.
\end{proposition}

\begin{proof}
  Par adjonction, on a
  \[
    \begin{split}
    \Hom_{\ooCat}(T, S_pC)
    & = \Hom_{\ooCat}(T, \Homidt(\Deltan{p}, C)) \\
    & \simeq \Hom_{\ooCat}(T \otimest \Deltan{p}, C)
    \end{split}
  \]
  et il s'agit en fait de montrer que $T \otimest \Deltan{p}$ est
  canoniquement isomorphe à $\Deltan{p} \wr T$. Puisque le foncteur $T
  \otimest {-}$ commute aux limites inductives,  que $\Deltan{p} \simeq
  \Dn{1} \amalg_{\Dn{0}} \dots \amalg_{\Dn{0}} \Dn{1}$ et
  que $T \otimest \Dn{0} \simeq \Dn{0}$ (car $T$ est connexe, voir
  le paragraphe~\ref{paragr:def_otimest}), on a
  \[
  T \otimest \Deltan{p} \simeq T \otimest \Dn{1} \amalg_{\Dn{0}} \cdots
  \amalg_{\Dn{0}} T \otimest \Dn{1}.
  \]
  Par ailleurs, en vertu du paragraphe~\ref{paragr:constr_W}, et avec ses
  notations, on a
  \[
    \Deltan{p} \wr T \simeq \Sigma'T \amalg_{\Dn{0}} \cdots \amalg_{\Dn{0}}
    \Sigma' T.
  \]
  Le résultat découle donc du lemme suivant :
\end{proof}

\begin{lemma}
  Soit $T$ une \oo-catégorie connexe. On a un isomorphisme canonique
  \[ T \otimest \Dn{1} \simeq \Sigma'T, \]
  où $\Sigma'T$ désigne la \oo-catégorie introduite au
  paragraphe~\ref{paragr:constr_W}, naturel en $T$.
\end{lemma}

\begin{proof}
  Par définition (et connexité de $T$), le carré
  \[
    \xymatrix{
      T \amalg T \ar[r] \ar[d] &
      \Dn{0} \amalg \Dn{0} \ar[d] \\ T \otimes \Dn{1} \ar[r] &
      \pushoutcorner T \otimest \Dn{1}
    }
  \]
  est cocartésien ; autrement dit, on a
  \[
    T \otimest \Dn{1} \simeq \Dn{0} \amalg_T (T \otimes \Dn{1})
    \amalg_T \Dn{0}.
  \]
  Le résultat découle donc de \cite[corollaire B.6.6]{AraMaltsiJoint} qui
  affirme, à une dualité près, que cette somme amalgamée est isomorphe à
  $\Sigma' T$.
\end{proof}

\begin{corollary}\label{coro:desc_SC}
  Soit $C$ une \oo-catégorie. Pour tout $p \ge 0$, on a un isomorphisme
  canonique
  \[
    S_pC \simeq \coprod_{c_0, \dots, c_p \in \Ob(C)} \Homi_C(c_0, c_1) \times \cdots
      \times \Homi_C(c_{p-1}, c_p)
  \]
  définissant un isomorphisme d'objets simpliciaux naturel en $C$.
  En particulier, si $C$ est une $n$-catégorie, avec $n \ge 1$, alors $S_pC$
  est une $(n-1)$\=/catégorie.
\end{corollary}

\begin{proof}
  Soit $T$ une \oo-catégorie connexe. En utilisant la
  proposition~\ref{prop:pu_SC}, on obtient
  \[
    \begin{split}
    \Hom_{\ooCat}(T, S_pC) & \simeq \Hom_{\ooCat}(\Deltan{p} \wr T, C) \\
    & \simeq \Hom_{\ooCat}(\Sigma'T \amalg_{\Dn{0}} \cdots \amalg_{\Dn{0}}
    \Sigma' T, C) \\
    & \simeq \Hom_{\ooCat}(\Sigma'T, C) \times_{\Ob(C)} \cdots
    \times_{\Ob(C)} \Hom_{\ooCat}(\Sigma'T, C).
    \end{split}
  \]
  Or, essentiellement par définition, la donnée d'un \oo-foncteur $\Sigma' T
  \to C$ correspond exactement à celle de deux objets $c$ et $c'$ de $C$ et
  d'un \oo-foncteur $T \to \Homi_C(c, c')$, d'où le résultat, par densité de
  la sous-catégorie pleine de $\ooCat$ formée des \oo-catégories connexes.
\end{proof}

\begin{paragraph}\label{paragr:def_Sn}
  Si $C$ est une \oo-catégorie, on notera $S^nC$ l'objet $n$-simplicial dans
  $\ooCat$ obtenu en itérant $n$ fois $S$, c'est-à-dire donné par
  \[
    \begin{split}
      {(\cDelta^n)}^\o & \to \ooCat \\
      (\Deltan{p_1}, \dots, \Deltan{p_n}) & \mapsto S_{p_n}\cdots S_{p_1}C.
    \end{split}
  \]
  En vertu du corollaire précédent, si $C$ est une $n$-catégorie, alors
  $S_{p_n}\cdots S_{p_1}C$ est un ensemble et on obtient donc un foncteur
  \[
    \begin{split}
      \nCat{n} & \to \pref{\cDelta^n} \\
      C & \mapsto S^nC.
    \end{split}
  \]
  Notons que lorsque $n = 1$, il résulte du corollaire précédent que ce
  foncteur n'est autre que le nerf usuel.
\end{paragraph}

\begin{proposition}\label{prop:lien_N_Nsn}
  Si $C$ est une $n$-catégorie, alors il existe un
  isomorphisme canonique d'ensembles $n$-simpliciaux
  \[ S^nC \simeq \Nsn{n}C, \]
  naturel en $C$.
\end{proposition}

\begin{proof}
  Si $p_1, \dots, p_n$ sont $n$ entiers, en utilisant la proposition~\ref{prop:pu_SC},
  on obtient
  \[
    \begin{split}
      (S^nC)_{p_1, \dots, p_n} & \simeq
      \Hom_{\ooCat}(\Dn{0}, (S^nC)_{p_1, \dots, p_n}) \\
      & \simeq \Hom_{\ooCat}(\Dn{0}, S_{p_n}\cdots S_{p_1}C) \\
      & \simeq \Hom_{\ooCat}(\Deltan{p_n} \wr \Dn{0}, S_{p_{n-1}}\cdots S_{p_1}C) \\
      & \simeq \Hom_{\ooCat}(\Deltan{p_n}, S_{p_{n-1}}\cdots S_{p_1}C) \\
      & \simeq \Hom_{\ooCat}(\Deltan{p_{n-1}} \wr \Deltan{p_n}, S_{p_{n-2}}\cdots S_{p_1}C) \\
      & \simeq \cdots \\
      & \simeq \Hom_{\ooCat}(\Deltan{p_1} \wr \cdots \wr \Deltan{p_n}, C) \\
      & \simeq (\Nsn{n}C)_{p_1, \dots, p_n},
    \end{split}
  \]
  d'où le résultat.
\end{proof}

On va maintenant produire un zigzag d'équivalences faibles simpliciales
entre l'ensemble simplicial $\delta^\ast NSC$ et le nerf de Street $NC$.

\begin{paragraph}
  Soit $C$ une \oo-catégorie. On va considérer son nerf de Street $NC$ comme
  un ensemble bisimplicial constant en la deuxième variable :
  \[
    N_{p,q}C = N_pC = \Hom_{\ooCat}(\On{p}, C).
  \]
  Notons qu'on a
  \[
      N_{p,q}C \simeq \Hom_{\ooCat}(\Dn{0} \otimest \On{p}, C).
  \]

  On va comparer cet ensemble bisimplicial avec l'ensemble bisimplicial
  $NSC$ dont l'ensemble des $(p, q)$-simplexes est
  \[
      (NS)_{p,q}C = N_qS_pC = \Hom_{\ooCat}(\On{q}, S_pC)
            = \Hom_{\ooCat}(\On{q}, \Homidt(\Deltan{p}, C)).
  \]
  Par adjonction, on a
  \[
    (NS)_{p,q}C \simeq \Hom_{\ooCat}(\On{q} \otimest \Deltan{p}, C).
  \]

  Afin de comparer ces deux ensembles bisimpliciaux, on introduit l'ensemble
  bisimplicial intermédiaire $XC$ défini par
  \[ X_{p,q}C = \Hom_{\ooCat}(\On{q} \otimest \On{p}, C). \]
\end{paragraph}

\begin{paragraph}\label{paragr:zigzag-1}
  Pour $p \ge 0$, le $1$-tronqué de $\On{p}$, c'est-à-dire la $1$-catégorie
  obtenue à partir de~$\On{p}$ en identifiant deux $1$-cellules lorsqu'elles
  sont reliées par un zigzag de $2$\=/cellules, est canoniquement isomorphe à
  $\Deltan{p}$. Ainsi, on dispose d'un \oo-foncteur canonique~$\On{p} \to
  \Deltan{p}$. Par ailleurs, pour $q \ge 0$, on a un unique \oo-foncteur
  $\On{q} \to \Dn{0}$. Ces deux \oo-foncteurs induisent, pour toute
  \oo-catégorie $C$, des applications
  \[
  \Hom_{\ooCat}(\Dn{0} \otimest \On{p}, C)
  \xto{}
  \Hom_{\ooCat}(\On{q} \otimest \On{p}, C)
  \xot{}
  \Hom_{\ooCat}(\On{q} \otimest \Deltan{p}, C)
  \]
  définissant des morphismes d'ensembles bisimpliciaux naturels en $C$.
  Ainsi, on dispose de morphismes d'ensembles bisimpliciaux
  \[ NC \xto{} XC \xot{} NSC \]
  naturels en $C$.
\end{paragraph}

\begin{proposition}\label{prop:zigzag_base}
  Pour toute \oo-catégorie $C$, les morphismes
  \[ NC \xto{} XC \xot{} NSC \]
  du paragraphe précédent sont des équivalences faibles diagonales.
\end{proposition}

\begin{proof}
  Commençons par montrer que le morphisme $NC \to XC$ est une équivalence
  faible diagonale. Fixons $q \ge 0$. En vertu du lemme bisimplicial, il
  suffit de montrer que le morphisme simplicial
  \[
    \Hom_{\ooCat}(\Dn{0} \otimest \On{\bullet}, C)
    \to
    \Hom_{\ooCat}(\On{q} \otimest \On{\bullet}, C)
  \]
  est une équivalence faible. Ce morphisme s'identifie par adjonction au
  morphisme
  \[
    \Hom_{\ooCat}(\On{\bullet}, \Homigt(\Dn{0}, C))
    \to
    \Hom_{\ooCat}(\On{\bullet}, \Homigt(\On{q}, C)),
  \]
  c'est-à-dire au nerf de Street du \oo-foncteur
  \[
    \Homigt(\Dn{0}, C)
    \to
    \Homigt(\On{q}, C)
  \]
  induit par le \oo-foncteur $\On{q} \to \Dn{0}$. Or, d'après la
  proposition~\ref{prop:On_contr}, ce dernier \oo-foncteur est la rétraction d'un
  rétracte par transformation lax (voir le
  paragraphe~\ref{paragr:retr_par_trans}) et, en vertu
  de la proposition~\ref{prop:trans_lax_Homit}, on obtient un rétracte par
  transformation oplax en appliquant le foncteur $\Homigt({-}, C)$. Ainsi,
  d'après le théorème~\ref{thm:nerf_retr_trans}, le nerf de Street de ce
  \oo-foncteur est une équivalence faible, ce qui entraîne que $NC \to XC$
  est bien une équivalence faible diagonale.

  Montrons maintenant que le morphisme $NSC \to XC$ est une équivalence
  faible diagonale. Fixons cette fois $p \ge 0$. Toujours en vertu
  du lemme bisimplicial, il suffit de montrer que le morphisme simplicial
  \[
    \Hom_{\ooCat}(\On{\bullet} \otimest \Deltan{p}, C)
    \to
    \Hom_{\ooCat}(\On{\bullet} \otimest \On{p}, C)
  \]
  est une équivalence faible. Par adjonction, ce morphisme
  s'identifie au morphisme
  \[
    \Hom_{\ooCat}(\On{\bullet}, \Homidt(\Deltan{p}, C))
    \to
    \Hom_{\ooCat}(\On{\bullet}, \Homidt(\On{p}, C)),
  \]
  qui n'est autre que le nerf de Street du \oo-foncteur
  \[
    \Homidt(\Deltan{p}, C) \to \Homidt(\On{p}, C)
  \]
  induit par le \oo-foncteur $\On{p} \to \Deltan{p}$. Or, d'après la
  proposition~\ref{prop:Op_Dp_retr}, ce dernier \oo-foncteur est la rétraction
  d'un rétracte par transformation oplax triviale sur les objets (voir le
  paragraphe~\ref{paragr:retr_par_trans}). Il résulte donc de la
  proposition~\ref{prop:trans_lax_Homit} que le
  \oo-foncteur obtenu en lui appliquant le foncteur $\Homigt({-}, C)$ est un
  rétracte par transformation lax. Son nerf de Street est donc une
  équivalence faible, toujours d'après le
  théorème~\ref{thm:nerf_retr_trans}, ce qui achève de montrer que $NSC \to
  XC$ est une équivalence faible diagonale.
\end{proof}

\begin{proposition}\label{prop:Street_bisimpl}
  Pour toute \oo-catégorie $C$, le nerf de Street $\NS C$ est naturellement
  faiblement équivalent à l'ensemble simplicial $\delta^\ast \NS SC$.
  Plus précisément, le zigzag du paragraphe~\ref{paragr:zigzag-1} induit un
  zigzag naturel d'équivalences faibles simpliciales entre $\NS C$ et
  $\delta^\ast\NS SC$.
\end{proposition}

\begin{proof}
  Puisque $\delta^\ast \NS C = \NS C$, cela résulte immédiatement de la
  proposition précédente.
\end{proof}

\begin{paragraph}\label{paragr:zigzag}
  Soit $0 < k < n$. Si $C$ est une \oo-catégorie, en appliquant le zigzag
  $N \to X \ot NS$, argument par argument, à la \oo-catégorie
  $(k-1)$\=/simpliciale~$S^{k-1}C$, on obtient un zigzag d'ensembles
  $(k+1)$-simpliciaux
  \[ NS^{k-1}C \xto{} XS^{k-1}C \xot{} NS^kC \pbox{.} \]
  L'ordre des indices simpliciaux est choisi de sorte que, en évaluant ce
  zigzag en $p_1, \dots, p_{k+1}$, on obtienne
  \[ N_{p_k}D \xto{}
  X^{}_{p_k,p_{k+1}}D \xot{} N_{p_{k+1}}S_{p_k}D \pbox{,} \]
  où
  \[ D = S_{p_{k-1}} \cdots S_{p_1}C. \]
  Considérons ces ensembles $(k+1)$-simpliciaux comme des ensembles
  $n$\=/simpliciaux constants en les $n - k - 1$ dernières variables. On
  peut ainsi composer ces zigzags et on obtient un zigzag d'ensembles
  $n$-simpliciaux, de longueur $2n-2$,
  \[ NC \to \cdots \ot NS^{n-1}C \pbox{,} \]
  où $NC$ est considéré comme constant en les $n-1$ dernières variables,
  naturel en $C$.
\end{paragraph}

\begin{proposition}
  Pour toute \oo-catégorie $C$, les morphismes $n$\=/simpliciaux du
  zigzag
  \[ NC \to \cdots \ot NS^{n-1}C \]
  du paragraphe précédent sont des équivalences faibles diagonales.
\end{proposition}

\begin{proof}
  Soit $k$ tel que $0 < k < n$. Il suffit de voir que les morphismes
  $(k+1)$-simpliciaux
  \[ NS^{k-1}C \xto{} XS^{k-1}C \xot{} NS^kC \]
  sont des équivalences faibles diagonales. Or, en évaluant ce zigzag en les
  $k-1$ premières variables en $p_1, \dots, p_{k-1}$, on obtient le zigzag
  bisimplicial
  \[ ND \xto{} XD \xot{} NSD \pbox{,} \]
  où
  \[ D = S_{p_{k-1}}\cdots S_{p_1}C. \]
  D'après la proposition~\ref{prop:zigzag_base}, les morphismes de ce
  dernier zigzag sont des équivalences faibles diagonales et on obtient
  donc le résultat en vertu du lemme multi-simplicial
  (lemme~\ref{lemme:multi_simpl}).
\end{proof}

\begin{theorem}\label{thm:comp_Street_n-simpl}
  Le nerf de Street, restreint aux $n$-catégories, et le nerf $n$\=/simplicial
  sont naturellement faiblement équivalents. Plus précisément, si $C$ est
  une $n$\=/catégorie, le zigzag du paragraphe~\ref{paragr:zigzag} induit un
  zigzag naturel d'équivalences faibles simpliciales entre $NC$ et
  $\delta^\ast\Nsn{n}C$.
\end{theorem}

\begin{proof}
  En vertu de la proposition précédente, le zigzag du
  paragraphe~\ref{paragr:zigzag} induit un zigzag d'équivalences faibles
  simpliciales entre $\delta^\ast NC = NC$ et~$\delta^\ast NS^{n-1}C$. Or, puisque
  $C$ est une $n$-catégorie, il résulte du corollaire~\ref{coro:desc_SC}
  que $S^{n-1}C$ est un objet $(n-1)$-simplicial en $1$-catégories et on a
  donc $NS^{n-1}C = S^n C$ d'après le paragraphe~\ref{paragr:def_Sn}. Ainsi,
  en vertu de la proposition
  \ref{prop:lien_N_Nsn}, on a
  \[ \delta^\ast NS^{n-1}C \simeq \delta^\ast S^nC \simeq \delta^\ast \Nsn{n}C, \]
  d'où le résultat.
\end{proof}

\section{Techniques de comparaison de nerfs}

\begin{paragraph}\label{paragr:comp}
  Soient
  \[ i : A \to \ooCat \quadet j : B \to \ooCat \]
  deux foncteurs de source des petites catégories.
  À ces données, on associe deux foncteurs « nerf »
  \[
    \begin{split}
      N_i : \ooCat & \to \pref{A} \\
     X \quad & \mapsto (a \mapsto \Hom_{\ooCat}(i(a), X)),
    \end{split}
    \quad
    \begin{split}
      N_j : \ooCat & \to \pref{B} \\
     X \quad & \mapsto (b \mapsto \Hom_{\ooCat}(j(b), X)).
    \end{split}
  \]
  Que signifie comparer ces deux foncteurs nerf ?

  Notons $\Hot$ la catégorie homotopique des espaces, c'est-à-dire la
  localisée de la catégorie des ensembles simpliciaux par les équivalences
  d'homotopie faibles. Pour toute petite catégorie $C$, on dispose d'un
  foncteur canonique
  \[ k_C : \pref{C} \to \Hot. \]
  Ce foncteur est défini comme le composé
  \[ \pref{C} \xto{i_C} \Cat \xto{N} \pref{\cDelta} \xto{p} \Hot \pbox{,} \]
  où $i_C$ est le foncteur catégorie des éléments, dont nous rappellerons
  la définition plus loin, $N$ est le nerf usuel et $p$ est le
  foncteur canonique vers la localisée.
  Comparer les foncteurs nerf $N_i$ et $N_j$, c'est montrer que les
  foncteurs composés
  \[
    \ooCat \xto{N_i} \pref{A} \xto{k_A} \Hot
    \quadet
    \ooCat \xto{N_j} \pref{B} \xto{k_B} \Hot
  \]
  sont isomorphes.

  Notons que si cette condition est satisfaite, alors les
  classes~$\W_i$ et~$\W_j$ des « équivalences faibles » de $\ooCat$ définies
  par ces foncteurs nerf, c'est-à-dire des \oo-foncteurs envoyés sur un
  isomorphisme par $k_AN_i$ et $k_BN_j$ respectivement, coïncident.
\end{paragraph}

Le but de cette section est de donner des critères de comparaison de
foncteurs nerf basés sur la théorie de l'homotopie de Grothendieck, et en
particulier sa théorie des structures d'asphéricité \cite[chapitre IV,
sections 72--78]{GrothPS}.

\medskip

Commençons par introduire les notions de base de cette théorie.

\begin{paragraph}
  Soit $u : A \to B$ un foncteur et soit $b$ un objet de $B$. On notera
  $\tr{A}{b}$ la catégorie dont les objets sont les couples $(a, f : u(a)
  \to b)$, où $a$ est un objet de $A$ et $f : u(a) \to b$ un morphisme de $B$, et dont les
  morphismes de $(a, f)$ vers $(a', f')$ sont les morphismes $g : a \to a'$
  de $A$ rendant commutatif le triangle évident, autrement dit, vérifiant
  $f'u(g) = f$.
\end{paragraph}

\begin{paragraph}
  Soit $A$ une petite catégorie. Si $F$ est un préfaisceau sur $A$, en
  appliquant la construction du paragraphe précédent au foncteur de Yoneda
  $A \hookto \pref{A}$, qu'on considérera toujours comme une inclusion, on
  obtient une catégorie $\tr{A}{F}$. C'est la \ndef{catégorie des éléments
  de $F$}. On obtient ainsi un foncteur
  \[
    \begin{split}
      i_A : \pref{A} & \to \Cat \\
      F & \mapsto \tr{A}{F}.
    \end{split}
  \]
  Explicitement, les objets de $\tr{A}{F}$ sont les couples $(a, f)$, où $a$
  est un objet de $A$ et $f : a \to F$ est un morphisme de préfaisceaux sur
  $A$, qu'on peut identifier par le lemme de Yoneda à un élément de $F(a)$.
\end{paragraph}

\begin{paragraph}\label{paragr:def_Thom}
  On dit qu'un foncteur $u : A \to B$ entre petites catégories est une
  \ndef{équivalence de Thomason} si son nerf $Nu : NA \to NB$ est une
  équivalence d'homotopie faible simpliciale.

  Une petite catégorie $A$ est dite \ndef{asphérique} si le foncteur de $A$
  vers la catégorie ponctuelle est une équivalence de Thomason ou, ce qui
  revient au même, si son nerf est un ensemble simplicial faiblement
  contractile.

  Un préfaisceau $F$ sur une petite catégorie $A$ est dit \ndef{asphérique}
  si sa catégorie des éléments $\tr{A}{F}$ est asphérique.

  Enfin, on dit qu'un morphisme de préfaisceaux $\phi : F \to G$ sur une petite
  catégorie~$A$ est une \ndef{équivalence faible} si le foncteur $i_A(\phi) :
  \tr{A}{F} \to \tr{A}{G}$ est une équivalence de Thomason.
\end{paragraph}

\begin{remark}
  La classe des équivalences d'homotopie faibles simpliciales étant
  fortement saturée (au sens où un morphisme est dans cette classe si et
  seulement si son image dans la localisée correspondante est inversible),
  il en est de même des classes des équivalences de Thomason et des
  équivalences faibles de préfaisceaux sur une petite catégorie $A$.
\end{remark}

\begin{remark}
  Les foncteurs entre petites catégories qui sont des adjoints, à gauche
  ou à droite, sont des équivalences de Thomason. Cela résulte immédiatement
  du fait qu'une transformation naturelle induit \forlang{via} le nerf une
  homotopie simpliciale. On en déduit qu'une catégorie admettant un objet
  initial ou final est asphérique et, par suite, que tout préfaisceau
  représentable est asphérique.
\end{remark}

\begin{remark}
  Si $A$ est une petite catégorie \emph{asphérique}, alors un
  préfaisceau~$F$ sur~$A$ est asphérique si et seulement si le morphisme $F
  \to e_{\pref{A}}$, où $e_{\pref{A}}$ désigne le préfaisceau final sur $A$,
  est une équivalence faible.
\end{remark}

\begin{remark}\label{rem:Illusie-Quillen}
  Un morphisme d'ensembles simpliciaux est une équivalence faible au sens du
  paragraphe~\ref{paragr:def_Thom} si et seulement si il est une équivalence
  d'homotopie faible simpliciale. Cela résulte du fait que, pour tout
  ensemble simplicial $X$, il existe un zigzag d'équivalences d'homotopie
  faibles entre $N(\tr{\cDelta}{X})$ et~$X$ (voir \cite[chapitre~VI,
  théorème~3.3.(ii)]{Illusie}). En particulier, un ensemble simplicial est
  asphérique si et seulement si il est faiblement contractile au sens
  classique.
\end{remark}

\begin{paragraph}
  Un foncteur $u : A \to B$ entre petites catégories est dit
  \ndef{asphérique} si, pour tout objet $b$ de $B$, la catégorie $\tr{A}{b}$
  est asphérique. Le théorème A de Quillen affirme précisément qu'un
  foncteur asphérique est une équivalence de Thomason.
\end{paragraph}

\begin{paragraph}\label{paragr:def_lambda}
  Soit $u : A \to B$ un foncteur entre petites catégories. Notons $u^\ast
  : \pref{B} \to \pref{A}$ le foncteur d'image inverse par $u$, c'est-à-dire de
  précomposition par $u$. Pour tout préfaisceau $F$ sur $B$, on dispose d'un
  foncteur
  \[
    \begin{split}
      \tr{A}{u^\ast(F)} & \to \tr{B}{F} \\
      (a, f) & \mapsto (u(a), f),
    \end{split}
  \]
  où $f$ est tantôt vu comme un morphisme $f : a \to u^\ast(F)$, tantôt
  comme un morphisme $f : u(a) \to F$, ces deux morphismes correspondant en vertu du
  lemme de Yoneda à un élément de $F(u(a))$. Ce foncteur est naturel en $F$
  et on dispose donc d'une transformation naturelle
  \[
    \shorthandoff{;}
    \xymatrix@C=2pc{
      \pref{B} \ar[rr]^{u^\ast}_(0.90){}="2" \ar[dr]_{i_B}_(0.55){}="1" & & \pref{A} \ar[dl]^{i_A} \\
      & \Cat
      \ar@{}"1";[ur]_(0.20){}="11"_(0.60){}="22"
      \ar@2"22";"11"_{\lambda_u}
      & \bpbox{.}
    }
  \]
\end{paragraph}

\begin{proposition}[Grothendieck]\label{prop:caract_asph}
  Un foncteur $u : A \to B$ entre petites catégories est asphérique si et
  seulement si la transformation naturelle $\lambda_u$ est une équivalence de
  Thomason argument par argument.
\end{proposition}

\begin{proof}
  C'est l'équivalence entre $(a)$ et $(c)$ de \cite[proposition
  1.2.9]{Maltsi}.
\end{proof}

Dans la suite de la section, on fixe une catégorie $\M$, que le lecteur peut
penser être $\ooCat$ ou $\nCat{n}$. Si $A$ est une petite catégorie et $i :
A \to \M$ un foncteur, on notera $N_i : \M \to \pref{A}$ le foncteur « nerf
» associé, défini comme au paragraphe \ref{paragr:comp}.

\begin{proposition}[Grothendieck]\label{prop:comp_nerf_asp}
  Soit
  \[
    \xymatrix@C=1.5pc{
      A \ar[dr]_i \ar[rr]^u & & B \ar[dl]^j \\
      & \M
    }
  \]
  un triangle commutatif de foncteurs, où $A$ et $B$ sont des petites
  catégories. Si $u$ est asphérique, alors les
  foncteurs composés
  \[ \M \xto{N_i} \pref{A} \xto{k_A} \Hot \quadet
  \M \xto{N_j} \pref{B} \xto{k_B} \Hot
  \]
  sont isomorphes. Plus précisément, la transformation naturelle $\lambda_u$
  induit une transformation naturelle entre
  \[
     \M \xto{N_i} \pref{A} \xto{i_A} \Cat
     \quadet
     \M \xto{N_j} \pref{B} \xto{i_B} \Cat
  \]
  qui est une équivalence de Thomason argument par argument.
\end{proposition}

\begin{proof}
  On dispose d'un diagramme
  \[
    \shorthandoff{;}
    \xymatrix@C=2pc{
      & \M \ar[dl]_{N_j} \ar[dr]^{N_i} \\
      \pref{B} \ar[rr]^{u^\ast}_(0.90){}="2" \ar[dr]_{i_B}_(0.55){}="1" & & \pref{A} \ar[dl]^{i_A} \\
      & \Cat
      \ar@{}"1";[ur]_(0.20){}="11"_(0.60){}="22"
      \ar@2"22";"11"_{\lambda_u}
    }
  \]
  dont le triangle supérieur est commutatif. En effet, si $X$ est dans $\M$
  et $a$ est dans $A$, on a
  \[ u^\ast N_j(X)_a = N_j(X)_{u(a)} = \Hom_\M(ju(a), X) = \Hom_\M(i(a), X)
  = N_i(X)_a. \]
  Puisque $u$ est asphérique, la transformation naturelle $\lambda_u$ est
  une équivalence de Thomason argument par argument et l'assertion suit.
\end{proof}

\begin{proposition}\label{prop:comp_nerf}
  Soient $i : A \to \M$ et $j : B \to \M$ deux foncteurs de sources des
  petites catégories. On suppose que pour tout objet $a$ de $A$ et tout
  objet $b$ de $B$, les préfaisceaux $N_ii(a)$ et $N_ij(b)$ sont des
  préfaisceaux asphériques de $\pref{A}$ et les préfaisceaux $N_ji(a)$ et
  $N_jj(b)$ sont des préfaisceaux asphériques de $\pref{B}$. Alors, les
  foncteurs composés
  \[ \M \xto{N_i} \pref{A} \xto{k_A} \Hot \quadet
  \M \xto{N_j} \pref{B} \xto{k_B} \Hot
  \]
  sont isomorphes.
  Plus précisément, il existe un zigzag de transformations naturelles de la
  forme
    \[ i_AN_i \tod \bullet \otd i_BN_j \]
  qui sont des équivalences de Thomason argument par argument.
\end{proposition}

\begin{proof}
  Soit $C$ la sous-catégorie pleine de $\M$ d'ensemble d'objets
  \[ \Ob(C) = \{i(a) \mid a \in \Ob(A)\} \cup \{j(b) \mid b \in \Ob(B)\}. \]
  Par définition, les foncteurs $i$ et $j$ se factorisent par $C$ et on
  dispose d'un diagramme commutatif
  \[
    \xymatrix{
      A \ar[dr]_i \ar[r]^{i'} & C \ar@{^{(}->}[d] & B \ar[l]_{j'} \ar[ld]^j \\
      & \M & \bpbox{.}
    }
  \]
  En vertu de la proposition précédente, il suffit de montrer que les
  foncteurs $i'$ et $j'$ sont asphériques et, par symétrie, il suffit de
  traiter le cas de $i'$. Soit donc $c$ un objet de $C$. Un objet de
  $\tr{A}{c}$ est un couple $(a, i'(a) \to c)$, avec $a$ un objet de $A$ et $i'(a)
  \to c$ un morphisme de $C$, c'est-à-dire un morphisme $i(a) \to c$ de $C$ ou, ce
  qui revient au même, un morphisme $a \to N_i(c)$ de préfaisceaux sur $A$.
  Autrement dit, on a $\tr{A}{c} \simeq \tr{A}{N_i(c)}$.  Ainsi, le foncteur
  $i'$ est asphérique si et seulement si, pour tout objet
  $c$ de $C$, le préfaisceau $N_i(c)$ sur $A$ est asphérique. Les objets de
  $c$ étant de la forme $i(a)$ ou $j(b)$, pour $a$ dans $A$ et $b$ dans $B$,
  l'assertion résulte de l'hypothèse que $N_ii(a)$ et $N_ij(b)$ sont
  asphériques.
\end{proof}

\begin{corollary}[Grothendieck]\label{coro:comp_nerf}
  Soient $i : A \to \M$ et $j : B \to \M$ deux foncteurs de sources des
  petites catégories. On suppose que pour tout objet $a$ de $A$ et tout
  objet $b$ de $B$, les préfaisceaux $N_i i(a)$ et $N_jj(b)$ sont
  asphériques. Alors les conditions suivantes sont équivalentes :
  \begin{enumerate}
    \item\label{item:a} les foncteurs $k_AN_i$ et $k_BN_j$ sont isomorphes ;
    \item\label{item:c} il existe un zigzag de transformations naturelles de la forme
    \[ i_AN_i \tod \bullet \otd i_BN_j \]
    qui sont des équivalences de Thomason argument par argument;
\end{enumerate}
  et elles entraînent la condition suivante :
  \begin{enumerate}[resume]
    \item\label{item:b} les foncteurs nerf $N_i$ et $N_j$ définissent les
    mêmes équivalences faibles sur $\M$ au sens où, pour tout morphisme $f$
    de $\M$, le morphisme $N_i(f)$ est une équivalence faible de
    préfaisceaux si et seulement si $N_j(f)$ en est une.
  \end{enumerate}
  De plus, si $A$ et $B$ sont asphériques et $\M$ possède un
  objet final, alors ces trois conditions sont équivalentes.
\end{corollary}

\begin{proof}
  Les implications $\ref{item:c} \tod \ref{item:a} \tod \ref{item:b}$ sont
  immédiates.

  Montrons l'implication $\ref{item:a} \tod \ref{item:c}$. En vertu de la
  proposition précédente, il suffit de montrer que pour tout objet~$a$
  de~$A$ et tout objet~$b$ de $B$, les préfaisceaux~$N_ji(a)$ et~$N_ij(b)$
  sont asphériques. Or, l'hypothèse \ref{item:a} entraîne immédiatement que
  si $X$ est un objet de $\M$, alors $N_iX$ est asphérique si et seulement
  si $N_jX$ est asphérique, d'où le résultat.

  Montrons enfin l'implication $\ref{item:b} \tod \ref{item:c}$ sous les
  hypothèses additionnelles. Soit $X$ un objet de $\M$. Pour les mêmes
  raisons que précédemment, il nous suffit de montrer que $N_iX$ est
  asphérique si et seulement si $N_jX$ est asphérique. Or, puisque la
  catégorie $A$ est asphérique, le préfaisceau $N_iX$ est asphérique si et
  seulement si le morphisme $N_iX \to e_{\pref{A}}$, où $e_{\pref{A}}$
  désigne le préfaisceau final, est une équivalence faible. Ce
  morphisme étant l'image par $N_i$ du morphisme $X \to e$, où $e$ désigne
  l'objet final de~$\M$, on obtient que le préfaisceau $N_iX$ est asphérique
  si et seulement si le morphisme $X \to e$ est envoyé par $N_i$ sur une
  équivalence faible. On conclut en utilisant le résultat analogue pour $N_j$.
\end{proof}

Enfin, terminons cette section par des rappels sur la notion de catégorie
totalement asphérique.

\begin{paragraph}
  Soit $A$ une petite catégorie asphérique. On dit que $A$ est
  \ndef{totalement asphérique} si le foncteur $i_A : \pref{A} \to \Cat$
  commute aux produits binaires à équivalence faible près au sens où,
  pour tous préfaisceaux $F$ et $G$ sur $A$, le
  foncteur canonique $\tr{A}{(F \times G)} \to \tr{A}{F} \times \tr{A}{G}$
  est une équivalence de Thomason. Cela entraîne en particulier que si $a_1$
  et $a_2$ sont deux objets de $A$, alors la catégorie $\tr{A}{(a_1 \times
  a_2)}$, où $a_1 \times a_2$ est vu comme un préfaisceau sur $A$, est
  asphérique. (Cette seconde condition, \forlang{a priori} plus faible, se
  trouve en fait être équivalente, voir par exemple \cite[proposition
  1.6.1]{Maltsi}.)
\end{paragraph}

\begin{theorem}[Cisinski-Maltsiniotis]\label{thm:theta_tot_asph}
  Pour tout $n \ge 0$, la catégorie $\Theta_n$ est totalement asphérique.
  De même, la catégorie $\Theta$ est totalement asphérique.
\end{theorem}

\begin{proof}
  Cela résulte de \cite[exemples 5.8 et 5.12]{CisMaltsiTheta}.
\end{proof}

\begin{remark}\label{rem:comp_Street_n-simpl}
  En particulier, la catégorie $\cDelta$, qui n'est autre que $\Theta_1$,
  est totalement asphérique. Il en résulte immédiatement que, pour tout $n
  \ge 0$, le foncteur diagonal $\delta : \cDelta \to \cDelta^n$ est
  asphérique. Ainsi, en vertu de la proposition~\ref{prop:caract_asph}, la
  transformation $\lambda_{\delta}$ est une équivalence de Thomason
  argument par argument, de sorte que les foncteurs
  \[
   \pref{\cDelta^n} \xto{\delta^\ast} \pref{\cDelta} \xto{k_\cDelta} \Hot
   \quadet
   \pref{\cDelta^n} \xto{k_{\cDelta^{\mkern-2mu n}}} \Hot
  \]
  sont isomorphes. Ainsi, le théorème~\ref{thm:comp_Street_n-simpl} entraîne
  que les foncteurs composés
  \[
    \nCat{n} \hookto \ooCat \xto{\NS} \pref{\cDelta} \xto{k_\cDelta} \Hot
    \quadet
    \nCat{n} \xto{\Nsn{n}} \pref{\cDelta^n} \xto{k_{\cDelta^{\mkern-2mu n}}} \Hot
  \]
  sont isomorphes.
\end{remark}

\vspace{-0.3cm}

\section{Comparaison du nerf de Street et du nerf cellulaire,\\ par
l'intermédiaire du nerf multi-simplicial}

Nous allons commencer par comparer le nerf $n$-simplicial et le nerf
$n$-cellulaire. Pour ce faire, nous allons montrer que le foncteur
$m_n : \cDelta^n \to \Theta_n$ du paragraphe~\ref{paragr:nerf_multi} est
asphérique. Cela résultera par récurrence du fait que, si $C$ est une
catégorie totalement asphérique, alors le foncteur $\mu : \cDelta \times C
\to \wrDelta{C}$ de ce même paragraphe est asphérique.

\begin{paragraph}
  Soit $C$ une catégorie. Considérons le triangle commutatif
  \[
    \xymatrix@C=1pc{
      \cDelta \times C \ar[dr]_{p_1} \ar[rr]^\mu & & \wrDelta{C} \ar[dl]^\pi \\
      & \cDelta & \bpbox{,}
    }
  \]
  où $p_1$ désigne la première projection et $\pi$ le foncteur du
  paragraphe~\ref{paragr:proj_wr}. Pour tout objet $(\Deltan{n}; c_1, \dots,
  c_n)$ de $\wrDelta{C}$, on déduit de ce triangle un foncteur
  \[ P : \tr{(\cDelta \times C)}{(\Deltan{n}; c_1, \dots, c_n)} \to
  \tr{\cDelta}{\Deltan{n}}. \]
  Explicitement, ce foncteur envoie un objet
  \[
  ((\Deltan{m}, c), (\phi; (f_{i',i})) : (\Deltan{m}; c, \dots, c) \to
  (\Deltan{n}; c_1, \dots, c_n))
  \]
 sur l'objet
  \[
  (\Deltan{m}, \phi : \Deltan{m} \to \Deltan{n})
  \]
  et un morphisme $(\psi, g)$ sur le morphisme $\psi$.
\end{paragraph}

\begin{proposition}\label{prop:P_asph}
  Pour toute catégorie $C$ et tout objet $(\Deltan{n}; c_1, \dots, c_n)$ de
  $\wrDelta{C}$, le foncteur
  \[ P : \tr{(\cDelta \times C)}{(\Deltan{n}; c_1, \dots, c_n)} \to
  \tr{\cDelta}{\Deltan{n}} \]
  du paragraphe précédent est une fibration de Grothendieck.
\end{proposition}

\begin{proof}
Le foncteur $p_1 : \cDelta \times C \to \cDelta$ étant une fibration de
Grothendieck, c'est un cas particulier du lemme général suivant.
\end{proof}

\begin{lemma}
  Soit
  \[
    \xymatrix@C=1pc{
      A \ar[dr]_v \ar[rr]^u & & B \ar[dl]^w \\
      & C
    }
  \]
  un triangle commutatif de foncteurs, avec $v$ une fibration de
  Grothendieck. Pour tout objet $b$ de~$B$, le foncteur canonique
  \[
    \begin{split}
      \tr{A}{b} & \to \tr{C}{w(b)} \\
      (a, f) & \mapsto (v(a), w(f))
    \end{split}
  \]
  est une fibration de Grothendieck.
\end{lemma}

\begin{proof}
  Le foncteur de l'énoncé se factorise comme le composé des foncteurs
  \[
    \begin{split}
      \tr{A}{b} & \to \tr{A}{w(b)} \\
      (a, f) & \mapsto (a, w(f))
    \end{split}
    \quadet
    \begin{split}
      \tr{A}{w(b)} & \to \tr{C}{w(b)} \\
      (a, g) & \mapsto (v(a), g)
    \end{split}
  \]
  et il suffit donc de montrer que chacun de ces foncteurs est une fibration
  de Grothendieck. On vérifie que toutes les flèches de $\tr{A}{b}$
  sont cartésiennes par rapport au premier foncteur et on en déduit
  immédiatement que ce premier foncteur est une fibration de Grothendieck.
  En ce qui concerne le second foncteur, il s'inscrit dans un carré
  cartésien
  \[
    \xymatrix{
      \tr{A}{w(b)} \ar[r] \ar[d] \pullbackcorner & A \ar[d]^v \\
      \tr{C}{w(b)} \ar[r] & C \pbox{,}
    }
  \]
  où les flèches horizontales sont les foncteurs d'oubli. Ainsi, ce second
  foncteur s'obtient par changement de base à partir de la fibration de
  Grothendieck $v$. C'est donc une fibration de Grothendieck, d'où le
  résultat.
\end{proof}

\begin{proposition}
  Pour toute petite catégorie totalement asphérique $C$, le foncteur
  \[ \mu : \cDelta \times C \to \wrDelta{C} \]
  est asphérique.
\end{proposition}

\begin{proof}
  Soit $(\Deltan{n}; c_1, \dots, c_n)$ un objet de $\wrDelta{C}$. Il s'agit
  de montrer que la catégorie
  $\tr{(\cDelta \times C)}{(\Deltan{n}; c_1, \dots, c_n)}$ est asphérique.
  En vertu de la proposition~\ref{prop:P_asph}, le foncteur
  \[ P : \tr{(\cDelta \times C)}{(\Deltan{n}; c_1, \dots, c_n)} \to
  \tr{\cDelta}{\Deltan{n}} \]
  est une fibration de Grothendieck. Si $(\Deltan{m}, \phi : \Deltan{m} \to
  \Deltan{n})$ est un objet de la catégorie but de $P$, on vérifie
  immédiatement que la fibre de $P$ en cet objet est la
  catégorie~$\tr{C}{\prod_{\phi(0) < i \le \phi(m)} c_i}$. La catégorie $C$
  étant supposée totalement asphérique, cette fibre est asphérique et le
  foncteur $P$ est donc une fibration à fibres asphériques, et par
  conséquent une équivalence de Thomason (voir \cite[§1, théorème A,
  corollaire]{QuillenHAKTI}).  Puisque la catégorie but de $P$ est
  asphérique (car elle possède un objet final), il en est de même de sa
  source, ce qu'on voulait montrer.
\end{proof}

\begin{corollary}
  Pour tout $n \ge 0$, le foncteur
  \[ m_n : \cDelta^n \to \Theta_n \]
  du paragraphe~\ref{paragr:nerf_multi} est asphérique.
\end{corollary}

\begin{proof}
  L'assertion est claire pour $n = 0$. Supposons-la vraie pour un
  entier~$n$.
  Par définition, le foncteur $m_{n+1}$
  est le composé
      \[
        \cDelta^{n+1} = \cDelta \times \cDelta^{n} \xto{\cDelta \times m_n} \cDelta \times
        \Theta_n \xto{\mu} \wrDelta{\Theta_n} = \Theta_{n+1} \pbox{.}
      \]
  Par récurrence, le foncteur $m_n$ est asphérique. Il en est donc de même
  du foncteur~$\cDelta \times m_n$, les foncteurs asphériques étant stables
  par produit (voir par exemple~\cite[corollaire~1.1.6]{Maltsi}). Par
  ailleurs, puisque la catégorie $\Theta_n$ est totalement asphérique
  d'après le théorème~\ref{thm:theta_tot_asph}, il résulte de la proposition
  précédente que le foncteur $\mu : \cDelta \times \Theta_n \to
  \wrDelta{\Theta_n}$ est asphérique. L'assertion est donc conséquence de la
  stabilité des foncteurs asphériques par composition (voir par
  exemple~\cite[proposition~1.1.8]{Maltsi}).
\end{proof}

\begin{theorem}\label{thm:comp_simpl_cell}
  Pour tout $n \ge 0$, le nerf $n$-simplicial et le nerf $n$-cellulaire sont
  équivalents au sens où les foncteurs composés
  \[
    \nCat{n} \xto{\Nsn{n}} \pref{\cDelta^n} \xto{k_{\cDelta^{\mkern-2mu n}}} \Hot
    \quadet
    \nCat{n} \xto{\NCn{n}} \pref{\Theta_n} \xto{k_{\Theta_n}} \Hot
  \]
  sont isomorphes. Plus précisément, la transformation naturelle
    \[ \lambda_{m_n} : i_{\cDelta^{\mkern-2mu n}}\Nsn{n} \tod
    i_{\Theta_n}\NCn{n} \]
  (voir le paragraphe~\ref{paragr:def_lambda}) est une équivalence de
  Thomason argument par argument.
\end{theorem}

\begin{proof}
  Cela résulte immédiatement du corollaire précédent, en vertu de la
  proposition~\ref{prop:comp_nerf_asp} appliquée au triangle commutatif
  \[
    \xymatrix@C=1.5pc{
      \cDelta^n \ar[dr]_{M_n} \ar[rr]^{m_n} & & \Theta_n \ar[dl]^{R_n} \\
      & \nCat{n}
    }
  \]
  (voir les paragraphes~\ref{paragr:def_R_n} et \ref{paragr:nerf_multi}).
\end{proof}

Nous allons maintenant comparer le nerf de Street et le nerf cellulaire.

\begin{lemma}
  Si $C$ est une $n$-catégorie, pour $n \ge 0$ un entier, alors l'ensemble
  simplicial~$\NS C$ est asphérique si et seulement si l'ensemble
  cellulaire $\NC C$ est asphérique.
\end{lemma}

\begin{proof}
  En vertu de la remarque~\ref{rem:comp_Street_n-simpl} (dont le contenu
  technique est le théorème~\ref{thm:comp_Street_n-simpl}), si $C$ est une
  $n$-catégorie, l'ensemble simplicial $\NS C$ est asphérique si et
  seulement si l'ensemble $n$-simplicial $\Nsn{n} C$ est asphérique. Par
  ailleurs, en vertu du théorème précédent, cet ensemble $n$-simplicial est
  asphérique si et seulement si l'ensemble $n$-cellulaire $\NCn{n} C$ est
  asphérique. Il nous suffit donc de justifier que ce dernier est asphérique
  si et seulement si l'ensemble cellulaire $\NC C$ est asphérique.

  Rappelons que le foncteur d'inclusion $\nCat{n} \hookto \ooCat$ admet un
  adjoint à gauche $\tau : \ooCat \to \nCat{n}$, le foncteur de
  \ndef{troncation intelligente}. Si $A$ est une \oo-catégorie, on a
    \[
      \tau(A)_k =
      \begin{cases}
        A_k & \text{si $0 \le k < n$,} \\
        A_n/{\sim} & \text{si $k = n$,}
      \end{cases}
    \]
    \renewcommand\taut{\tau^{}_\Theta}%
    \newcommand\tautast{\tau^\ast_\Theta}%
    où la relation $\sim$ identifie deux $n$-cellules reliées par un zigzag
    de $(n+1)$-cellules. Puisque le foncteur $\tau$ est un adjoint à gauche, il
    commute aux limites inductives et il résulte de la description des objets
    de $\Theta$ comme sommes amalgamées que le foncteur~$\tau$
    envoie~$\Theta$ dans~$\Theta_n$ et induit donc un foncteur $\taut
    : \Theta \to \Theta_n$, adjoint à gauche du foncteur d'inclusion
    $\Theta_n \hookto \Theta$. Le foncteur $\taut$ étant un adjoint à gauche, il
    est asphérique et, en vertu de la proposition~\ref{prop:caract_asph},
    pour tout ensemble $n$\=/cellulaire~$X$, le foncteur $\lambda_{\taut} :
    \tr{\Theta}{\tautast(X)} \to
    \tr{\Theta_n}{X}$ est une équivalence de Thomason. Ainsi, si $C$ est une
    $n$\=/catégorie, l'ensemble $n$\=/cellulaire~$\NCn{n}C$ est asphérique si et
    seulement si l'ensemble cellulaire~$\tautast(\NCn{n}C)$ est asphérique.
    Or, il est immédiat, par adjonction, que ce dernier ensemble cellulaire
    est canoniquement isomorphe au nerf cellulaire~$\NC C$ de $C$, d'où le
    résultat.
\end{proof}

\begin{theorem}\label{thm:comp_Street_cellulaire}
  Le nerf de Street et le nerf cellulaire sont équivalents au sens où les
  foncteurs composés
  \[
    \ooCat \xto{\NS} \pref{\cDelta} \xto{k_{\cDelta}} \Hot
    \quadet
    \ooCat \xto{\NC} \pref{\Theta} \xto{k_{\Theta}} \Hot
  \]
  sont isomorphes. Plus précisément, il existe un zigzag de transformations
  naturelles de la forme
    \[ i_{\cDelta}\NS \tod \bullet \otd i_{\Theta}\NC \]
  qui sont des équivalences de Thomason argument par argument.
\end{theorem}

\begin{proof}
  On va appliquer la proposition~\ref{prop:comp_nerf}. Soit $n \ge 0$
  et soit $S$ un objet de $\Theta$. Il nous faut montrer que les
  ensembles simpliciaux $\NS \On{n}$ et $\NS S$, et les ensembles
  cellulaires $\NC\On{n}$ et $\NC S$, sont asphériques. En vertu du lemme
  précédent, cela équivaut à montrer que $\NS\On{n}$ et~$\NC S$ sont
  asphériques. Or, l'ensemble simplicial~$\NS \On{n}$ est contractile en
  vertu de la proposition~\ref{prop:On_contr}, et est donc asphérique. Quant à
  l'ensemble cellulaire $\NC S$, le nerf cellulaire étant pleinement fidèle,
  cet ensemble cellulaire est isomorphe au préfaisceau représentable $S$, qui est
  asphérique, d'où le résultat.
\end{proof}

\begin{corollary}\label{coro:Thom_cell}
  Un \oo-foncteur $u : A \to B$ est une équivalence de Thomason si et
  seulement si son nerf cellulaire $\NC u : \NC A \to \NC B$ est une
  équivalence faible d'ensembles cellulaires, au sens du
  paragraphe~\ref{paragr:def_Thom}.
\end{corollary}

\begin{proof}
  Cela résulte immédiatement du théorème précédent.
\end{proof}

\begin{remark}
  Dans \cite{Gagna}, Gagna montre que le nerf de Street $\NS : \ooCat \to
  \pref{\cDelta}$ induit une équivalence de catégories $\Ho(\ooCat) \to
  \Hot$ (voir son théorème 5.6), où $\Ho(\ooCat)$ désigne la localisation de
  la catégorie $\ooCat$ par la classe des équivalences de
  Thomason, c'est-à-dire des \oo-foncteurs dont le nerf de Street est une
  équivalence faible simpliciale. Il résulte donc de notre
  théorème~\ref{thm:comp_Street_cellulaire} que le nerf cellulaire $\NC :
  \ooCat \to \pref{\Theta}$, postcomposé par le foncteur $k_\Theta :
  \pref{\Theta} \to \Hot$, induit la même équivalence de catégories
  $\Ho(\ooCat) \to \Hot$, à isomorphisme près. De plus, la
  catégorie~$\Theta$ étant test au sens de Grothendieck (voir
  \cite[paragraphe~2.5 et exemple~5.12]{CisMaltsiTheta}), ce dernier
  foncteur induit une équivalence de catégories $\Ho(\pref{\Theta}) \to
  \Hot$, où $\Ho(\pref{\Theta})$ désigne la catégorie des ensembles
  cellulaires localisée par rapport aux équivalences faibles d'ensembles
  cellulaires, et par conséquent le nerf cellulaire induit une équivalence
  de catégories~$\Ho(\ooCat) \to \Ho(\pref{\Theta})$.

  En fait, le caractère fonctoriel de la preuve de Gagna permet d'obtenir un
  résultat plus fort : le nerf de Street induit une équivalence de $(\infty,
  1)$-catégories faibles ou de dérivateurs $\Ho(\ooCat) \to \Hot$. On en
  déduit immédiatement qu'il en est de même du nerf cellulaire.

  Par ailleurs, pour $n \ge 1$, Gagna montre que le résultat analogue pour
  la restriction \smash{$\nCat{n} \hookto \ooCat \xto{\NS} \pref{\cDelta}$} du nerf
  de Street à $\nCat{n}$ est également vrai (voir toujours son théorème 5.6). On
  obtient ainsi, en vertu de nos théorèmes \ref{thm:comp_Street_n-simpl} et
  \ref{thm:comp_simpl_cell} (voir également la
  remarque~\ref{rem:comp_Street_n-simpl}), que les foncteurs nerf de Street,
  nerf $n$-simplicial et nerf $n$-cellulaire,
  \[
    \nCat{n} \hookto \ooCat \xto{\NS} \pref{\cDelta},
    \quad
    \Nsn{n} : \nCat{n} \to \pref{\cDelta^n},
    \quad
    \NCn{n} : \nCat{n} \to \pref{\Theta_n},
  \]
  induisent tous les trois la même équivalence de $(\infty,
  1)$-catégories faibles ou de dérivateurs~$\Ho(\nCat{n}) \to \Hot$.
\end{remark}

\section{Deux applications}

\begin{paragraph}
  On dira qu'un \oo-foncteur $u : C \to D$ est une \ndef{équivalence de
  Thomason} si son nerf de Street $Nu : NC \to ND$ est une équivalence
  faible simpliciale.
\end{paragraph}

\begin{theorem}
  Soit $u : C \to D$ un \oo-foncteur. On suppose que
  \begin{enumerate}
    \item $u$ induit une bijection $\Ob(C) \xto{\sim} \Ob(D)$ ;
    \item pour tous objets $c$ et $c'$ de $C$, le \oo-foncteur
        \[ \Homi_C(c, c') \to \Homi_D(u(c), u(c')) \]
        induit par $u$ est une équivalence de Thomason.
  \end{enumerate}
  Alors $u$ est une équivalence de Thomason.
\end{theorem}

\begin{proof}
  Considérons le morphisme $Su : SC \to SD$ d'objets simpliciaux dans
  $\ooCat$, où $S$ est la construction du paragraphe~\ref{paragr:def_SC}. En
  vertu de la description explicite de cette construction donnée par le
  corollaire~\ref{coro:desc_SC}, la stabilité des équivalences de Thomason
  par produit fini et somme (par connexité des orientaux) entraîne que ce
  morphisme est une équivalence de Thomason argument par argument. Autrement
  dit, le morphisme bisimplicial $\NS Su : \NS SC \to \NS SD$ est une
  équivalence faible lorsque l'on fixe le premier argument et le lemme
  bisimplicial implique donc que $\delta^\ast \NS Su$ est une équivalence
  faible simpliciale. Or, en vertu de la
  proposition~\ref{prop:Street_bisimpl}, le nerf de Street~$\NS u$ est une
  équivalence faible si et seulement si $\delta^\ast \NS Su$ est une
  équivalence faible, d'où le \nohyphen{résultat}.
\end{proof}

\begin{corollary}
  Soit $u : C \to D$ un \oo-foncteur et soit $n \ge 0$. On suppose que
  \begin{enumerate}
    \item pour tout $i$ tel que $0 \le i \le n$, le \oo-foncteur $u$ induit
    une bijection $C_i \to D_i$ ;
    \item pour tout couple $x$ et $y$ de $n$-cellules parallèles de $C$, le \oo-foncteur
        \[ \Homi_C(x, y) \to \Homi_D(u(x), u(y)) \]
        induit par $u$ est une équivalence de Thomason.
  \end{enumerate}
  Alors $u$ est une équivalence de Thomason.
\end{corollary}

\begin{proof}
  Le théorème précédent est précisément le cas $n = 0$ de cet énoncé. Le cas
  général s'en déduit par récurrence. En effet, si $n > 0$, pour conclure,
  en appliquant le théorème précédent, il suffit de montrer que, pour tous
  objets~$c$ et~$c'$ de~$C$, le \oo-foncteur $\Homi_C(c, c') \to \Homi_D(u(c),
  u(c'))$ est une équivalence de Thomason. Or, l'hypothèse sur $u$ entraîne
  immédiatement que l'application $\Homi_C(c, c')_i \to \Homi_D(u(c),
  u(c'))_i$, qui est induite par l'application $C_{i+1} \to D_{i+1}$, est
  bijective pour $0 \le i \le n - 1$, ainsi
  que le fait que, pour $x$ et $y$ deux $(n-1)$-cellules parallèles de
  $\Homi_C(c, c')$, le \oo-foncteur
  \[  \Homi_{\Homi_C(c, c')}(x, y) \to \Homi_{\Homi_D(u(c), u(c'))}(u(x),
  u(y)),
  \]
  qui n'est autre que le \oo-foncteur $\Homi_C(x, y) \to \Homi_D(u(x),
  u(y))$, est une équivalence de Thomason. On conclut donc en appliquant
  l'hypothèse de récurrence au rang~$n - 1$.
\end{proof}

\begin{paragraph}
  Soit $J$ un sous-ensemble de $\N\sauf\{0\}$. Si $C$ est une \oo-catégorie,
  on notera $D_J(C)$ la \oo-catégorie obtenue à partir de $C$ en inversant
  l'orientation des $j$-cellules pour $j$ dans $J$. Cette construction
  définit un foncteur $D_J : \ooCat \to \ooCat$ qui est un automorphisme
  involutif de $\ooCat$.
\end{paragraph}

\begin{lemma}
  Soit $J$ une partie de $\N\sauf\{0\}$. Alors, pour tout objet $S$ de
  $\Theta$, la \oo-catégorie $D_J(S)$ est dans $\Theta$.
\end{lemma}

\begin{proof}
  Soit $S$ un objet de $\Theta_n$. On va démontrer le résultat par
  récurrence sur $n$. Si $n = 0$, alors $S = \Dn{0}$ et $D_J(\Dn{0}) =
  \Dn{0}$. Supposons $n \ge 1$. En vertu des
  paragraphes~\ref{paragr:constr_W} et~\ref{paragr:def_R_n}, la
  $n$-catégorie $S$ est la limite inductive d'un diagramme
  \[
    \xymatrix@C=1pc{
      \Sigma'T_n & & \Sigma'T_{n-1} & & \cdots & & \Sigma'T_1 \\
                   & \Dn{0} \ar[lu]^{0} \ar[ru]_{1} & & \Dn{0} \ar[lu]^{0}
                   & \cdots & \Dn{0} \ar[ur]_1 & \bpbox{,}\\
    }
  \]
  où $T_1, \dots, T_n$ sont des objets de $\Theta_{n-1}$, ce qu'on écrira
  \[
    S \simeq \Sigma'T_n \amalg_{\Dn{0}} \dots \amalg_{\Dn{0}} \Sigma'T_1.
  \]
  Si $J'$ est une partie de $\N \sauf \{0, 1\}$, il est immédiat que pour
  toute \oo-catégorie $C$, on a
  \[ D_{J'}(\Sigma'C) = \Sigma'(D_{J'-1}(C)), \]
  où $J' - 1 = \{ j - 1 \mid j \in J'\}$. Par ailleurs, on a
  \[ D_{\{1\}}(\Sigma' C) \simeq \Sigma' C, \]
  mais les \oo-foncteurs $0 , 1 : \Dn{0} \to \Sigma' C$ sont échangés à
  travers cet isomorphisme.
  Ainsi, en posant $K = \{j - 1 \mid j \in J, \, j > 1\}$ et
  en utilisant le fait que $D_J$ commute aux limites inductives, on obtient,
  si $1$ n'appartient pas à $J$,
  \[
    D_{J}(S) \simeq \Sigma'(D_K(T_n)) \amalg_{\Dn{0}} \dots \amalg_{\Dn{0}}
    \Sigma'(D_K(T_1))
  \]
  et, si $1$ appartient à $J$,
  \[
    D_J(S) \simeq \Sigma'(D_K(T_1)) \amalg_{\Dn{0}} \dots \amalg_{\Dn{0}}
    \Sigma'(D_K(T_n)),
  \]
  d'où le résultat par récurrence.
\end{proof}

\begin{theorem}\label{thm:dual_Thomason}
  Soit $u : C \to D$ un \oo-foncteur et soit $J$ une partie de
  $\N\sauf\{0\}$. Alors $u$ est une équivalence de Thomason si et seulement
  si $D_J(u)$ en est une.
\end{theorem}

\begin{proof}
  En vertu du corollaire~\ref{coro:Thom_cell}, il suffit de montrer que $\NC
  u$ est une équivalence faible cellulaire si et seulement si $\NC D_J(u)$ en
  est une. D'après le lemme
  précédent, la dualité~$D_J$ induit un automorphisme de $\Theta$, qu'on
  notera $d_J$. Le foncteur~$d_J$ étant un isomorphisme, il est asphérique
  et on peut donc appliquer la proposition~\ref{prop:comp_nerf_asp} au
  triangle commutatif
  \[
    \xymatrix@C=1.5pc{
      \Theta \ar[dr]_{d_JR} \ar[rr]^{d_J} & & \Theta \ar[dl]^R \\
      & \ooCat & \pbox{,}
    }
  \]
  où $R$ est le foncteur induisant le nerf cellulaire. Or, le foncteur
  nerf $N_{Rd_J}$ associé à~$Rd_J$ n'est autre que $\NC D_J$ puisque, si $A$
  est une \oo-catégorie et $S$ un objet de $\Theta$, on~a
  \[
    \begin{split}
    N_{Rd_J}(A)_S
    & = \Hom_{\ooCat}(Rd_J(S), A)
    = \Hom_{\ooCat}(D_JR(S), A) \\
    & \simeq \Hom_{\ooCat}(R(S), D_J(A))
    = \NC(D_J(A))_S.
    \end{split}
  \]
  La proposition~\ref{prop:comp_nerf_asp} entraîne ainsi que les foncteurs
  $k_\Theta \NC$ et $k_\Theta \NC D_J$ sont isomorphes et donc le résultat.
\end{proof}

\begin{corollary}
  Soit $J$ une partie de $\N\sauf\{0\}$. Le foncteur
  \[
    \ooCat \xto{D_J} \ooCat \xto{\NS} \pref{\cDelta} \xto{k_{\cDelta}} \Hot
  \]
  est isomorphe au foncteur
  \[
    \ooCat \xto{\NS} \pref{\cDelta} \xto{k_{\cDelta}} \Hot \pbox{.}
  \]
  Plus précisément, il existe un zigzag de transformations naturelles de la
  forme
    \[ i_{\cDelta}\NS D_J \tod \bullet \otd i_{\cDelta}\NS \]
  qui sont des équivalences de Thomason argument par argument.

  En particulier, si $C$ est une \oo-catégorie, alors les ensembles
  simpliciaux $\NS C$ et~$\NS(D_J(C))$ ont même type d'homotopie.
\end{corollary}

\begin{proof}
  Il est immédiat que le foncteur $ND_J$ est le foncteur nerf associé à
  l'objet cosimplicial
  \[
    \cO_J : \cDelta \xto{\cO} \ooCat \xto{D_J} \ooCat \pbox{.}
  \]
  Ainsi, le théorème précédent affirme précisément que les foncteurs nerf
  associés aux objets cosimpliciaux $\cO$ et $\cO_J$ définissent les mêmes
  équivalences faibles sur $\ooCat$ au sens de la condition~\ref{item:b} du
  corollaire~\ref{coro:comp_nerf}. Le résultat est donc conséquence de ce
  même corollaire.
\end{proof}


Dans \cite{AraMaltsiThmAI} et \cite{AraMaltsiThmAII}, nous avons établi un
théorème A de Quillen \oo-catégorique pour les tranches au-dessous. (On
renvoie à la section 6 de \cite{AraMaltsiJoint} pour la définition des
tranches et en particulier au paragraphe 6.31 et à la remarque 6.37.) Les
tranches au-dessus s'obtenant par dualité à partir des tranches au-dessous,
on en déduit, en utilisant le théorème précédent, un théorème A pour les
tranches au-dessus :

\begin{coro}
  Soit
  \[
    \shorthandoff{;}
    \xymatrix@C=1.5pc{
      A \ar[rr]^u \ar[dr]_(0.40){v}_(.60){}="f" & & B \ar[dl]^(0.40){w} \\
      & C
      \ar@{}"f";[ur]_(.15){}="ff"
      \ar@{}"f";[ur]_(.55){}="oo"
      \ar@<-0.0ex>@2"oo";"ff"_\alpha
      &
    }
  \]
  un triangle de \oo-foncteurs commutatif à une transformation oplax $\alpha$
  près. Si pour tout objet $c$ de $C$, le \oo-foncteur canonique
  \smash{$\tr{(u, \alpha)}{c} : \tr{A}{c} \to \tr{B}{c}$} est une
  équivalence de Thomason, alors il en est de même de $u$.
\end{coro}

\begin{proof}
  Pour $D$ une \oo-catégorie, on notera $D^\o$ le dual total de $D$,
  c'est-à-dire la \oo-catégorie $D_{\N \sauf \{0\}}(C)$. Si $\beta : \Dn{1}
  \otimes D \to E$ est une transformation oplax d'un \oo-foncteur $f : D \to
  E$ vers un \oo-foncteur $g : D \to E$, alors, en vertu de l'isomorphisme
  canonique $(\Dn{1} \otimes D)^\o \simeq \Dn{1}^\o \otimes D^\o$ (voir
  \cite[proposition~A.22]{AraMaltsiJoint}), le \oo-foncteur~$\beta^\o$
  définit une transformation oplax de $g^\o : D^\o \to E^\o$ vers $f^\o :
  D^\o \to E^\o$. Ainsi, en appliquant la dualité totale au
  triangle de l'énoncé, on obtient un triangle
  \[
    \shorthandoff{;}
    \xymatrix@C=1.5pc{
      A^\o \ar[rr]^{u^\o} \ar[dr]_(0.40){v^\o}_(.60){}="g" & & B^\o
      \ar[dl]^(0.40){w^\o} \\
      & C^\o
      \ar@{}"g";[ur]_(.20){}="gg"
      \ar@{}"g";[ur]_(.50){}="oo"
      \ar@<-0.0ex>@2"gg";"oo"^{\alpha^\o}
      & \pbox{,}
    }
  \]
  commutant à une transformation oplax près. Par ailleurs, en vertu de
  \cite[remarque~6.37]{AraMaltsiJoint}, si $D$ est une \oo-catégorie
  au-dessus de $C$, alors, pour tout objet $c$ de $C$, on a un isomorphisme
  canonique $(\tr{D}{c})^\o \simeq \cotr{D^\o}{c}$. De plus, le \oo-foncteur
  \smash{$(\tr{(u, \alpha)}{c})^\o : (\tr{A}{c})^\o \to (\tr{B}{c})^\o$}
  s'identifie au \oo-foncteur
  \smash{$\cotr{(u^\o, \alpha^\o)}{c} : \cotr{A^\o}{c} \to \cotr{B^\o}{c}$}.
  Ainsi, puisque, en vertu du théorème précédent, les équivalences de Thomason
  sont stables par la dualité totale, l'hypothèse de l'énoncé est
  équivalente au fait que, pour tout objet $c$ de $C$, le \oo-foncteur
  \smash{$\cotr{(u^\o, \alpha^\o)}{c} : \cotr{A^\o}{c} \to \cotr{B^\o}{c}$}
  est une équivalence de Thomason, c'est-à-dire à l'hypothèse du théorème A
  pour les tranches au-dessous (voir \cite[théorème~7.8]{AraMaltsiThmAII}).
  On peut donc appliquer ce théorème A et on obtient que le \oo-foncteur
  $u^\o$ est une équivalence de Thomason. En appliquant de nouveau le
  théorème précédent, on en déduit que $u$ est une équivalence de Thomason,
  ce qu'on voulait démontrer.
\end{proof}

De même, dans \cite{AraThmB}, le premier auteur a démontré un théorème B
de Quillen \oo-catégorique pour les tranches au-dessous. On en déduit un
théorème B pour les tranches au-dessus :

\begin{corollary}
  Soit $u : A \to B$ un \oo-foncteur tel que, pour toute $1$-cellule $f : b
  \to b'$ de $B$, le \oo-foncteur canonique $\tr{A}{f} : \tr{A}{b} \to \tr{A}{b'}$
  soit une équivalence de Thomason. Alors, pour tout objet $b$ de $B$,
  le carré cartésien
  \[
    \xymatrix{
      \tr{A}{b} \pullbackcorner \ar[d]_{\tr{u}{b}} \ar[r] & A \ar[d]^u \\
      \tr{B}{b} \ar[r] & B \pbox{,}
    }
  \]
  où les flèches horizontales sont les \oo-foncteurs d'oubli, est un carré
  homotopiquement cartésien.
\end{corollary}

\begin{proof}
  Le résultat se déduit de l'énoncé analogue pour les tranches au-dessous
  (voir \cite[corollaire 3.11]{AraThmB}) en utilisant le
  théorème \ref{thm:dual_Thomason}, comme dans la preuve précédente. Plus
  précisément, on applique cet énoncé au foncteur $u^\o$, les hypothèses
  étant vérifiées en vertu de la stabilité des équivalences de Thomason par
  dualité, et on obtient que le carré qui se déduit du carré de l'énoncé en
  appliquant la dualité totale est homotopiquement cartésien. Or,
  le foncteur dual total étant une involution préservant les équivalences
  faibles, il respecte les carrés homotopiquement cartésiens, et le carré de
  l'énoncé est donc bien homotopiquement cartésien.
\end{proof}

\begin{remark}
  Les deux corollaires précédents sont des exemples de résultats qui se
  déduisent de le stabilité des équivalences de Thomason par dualité.
  Dans \cite{AraMaltsiThmAII} et \cite{AraThmB}, nous avons annoncé un
  certain nombre d'autres tels résultats ; ceux-ci se démontrent tous de
  manière similaire aux deux corollaires précédents. On renvoie en
  particulier aux remarques 5.28 et 7.13 de \cite{AraMaltsiThmAII} et aux
  remarques 3.16 et 4.5 de \cite{AraThmB}.
\end{remark}

\appendix

\section{Rétractes par transformation}
\label{app}

Le but de cet appendice est de rappeler les définitions et quelques
propriétés des rétractes par transformation lax et oplax, et de démontrer un
des points clés de la preuve de la proposition~\ref{prop:zigzag_base}, à
savoir que le \oo-foncteur canonique $\On{n} \to \Deltan{n}$ qui y est
considéré est la rétraction d'un rétracte par transformation oplax triviale
sur les objets.

\begin{paragraph}\label{paragr:retr_par_trans}
  Soit $i : A \to B$ un \oo-foncteur. On dit que $i$ est un \ndef{rétracte par
  transformation oplax} si $i$ admet une rétraction $r : B \to A$
  (c'est-à-dire un \oo-foncteur vérifiant~$ri = \id{A}$) pour laquelle il
  existe une transformation oplax $\alpha$ entre $ir$ et $\id{B}$. Dans une
  telle situation, on dit aussi que $r$ est la \ndef{rétraction d'un
  rétracte par transformation oplax}. On parle de \ndef{rétracte par
  transformation oplax triviale sur les objets} si la transformation oplax
  entre $ir$ et $\id{B}$ peut être choisie triviale sur les objets (au sens
  introduit au paragraphe~\ref{paragr:desc_Homit}).

  On définit des variantes lax de ces notions en remplaçant la
  transformation oplax~$\alpha$ par une transformation lax.
\end{paragraph}

\begin{theorem}\label{thm:nerf_retr_trans}
  Soit $i : A \to B$ un rétracte par transformation lax ou oplax. Alors
  $Ni : NA \to NB$ est un rétracte par déformation simplicial et en
  particulier une équivalence faible.
\end{theorem}

\begin{proof}
  Cela résulte immédiatement de \cite[théorème A.11]{AraMaltsiThmAII} dans
  le cas oplax et de \cite[paragraphe~A.18]{AraMaltsiThmAII} dans le cas lax.
\end{proof}

La preuve de la proposition~\ref{prop:zigzag_base} s'appuie sur l'existence
de deux rétractes par transformation :

\begin{proposition}\label{prop:On_contr}
  Pour tout $n \ge 0$, l'unique \oo-foncteur $\On{n} \to \Deltan{0}$ est la
  rétraction d'un rétracte par transformation lax. En particulier, l'ensemble
  simplicial $N \On{n}$ est contractile.
\end{proposition}

\begin{proof}
  L'assertion sur le \oo-foncteur $\On{n} \to \Deltan{0}$ résulte de
  \cite[proposition~5.16]{AraMaltsiThmAII} et de \cite[remarque
  B.4.11]{AraMaltsiJoint}. Le fait que $N \On{n}$ est contractile est alors
  conséquence du théorème précédent.
\end{proof}

\begin{proposition}\label{prop:Op_Dp_retr}
  Pour tout $n \ge 0$, le \oo-foncteur canonique $\On{n} \to \Deltan{n}$
  (voir le paragraphe~\ref{paragr:zigzag-1}) est la rétraction d'un
  rétracte par transformation oplax triviale sur les objets.
\end{proposition}

\begin{proof}
  Appelons $r$ le \oo-foncteur de l'énoncé. Ce foncteur admet une unique
  section $s : \Deltan{n} \to \On{n}$. Explicitement, celle-ci envoie,
  pour $0 \le i < n$, l'unique flèche de $i$ vers $i+1$ dans $\Deltan{n}$
  sur l'unique $1$-cellule de $i$ vers $i+1$ dans~$\On{n}$. Nous allons
  construire une transformation oplax $\alpha : \id{\On{n}} \tod sr$.
  Puisque pour tout objet $i$ de~$\On{n}$, l'unique $1$-cellule de $i$ vers
  $i$ est $\id{i}$, une telle transformation sera nécessairement triviale
  sur les objets.

  Pour construire $\alpha$, nous allons utiliser la théorie des complexes
  dirigés augmentés de Steiner \cite{Steiner}. Rappelons qu'un \ndef{complexe
  dirigé augmenté} est un complexe de chaînes~$C$ de groupes abéliens en
  degrés positifs, muni d'une augmentation et, pour tout $p \ge 0$, d'un
  sous-monoïde $C_p^\ast$ de $C_p^{}$ des $p$-chaînes \ndef{positives}. Les
  \ndef{morphismes de complexes dirigés augmentés} sont les morphismes de complexes
  de chaînes augmentés qui envoient les chaînes positives sur des chaînes
  positives. On obtient ainsi une catégorie~$\Cda$. Dans~\cite{Steiner},
  Steiner construit un foncteur de « linéarisation » $\lambda : \ooCat \to
  \Cda$ dont nous ne rappellerons pas la définition générale.
  Notons simplement que dans le cas d'une \oo-catégorie $C$ libre au sens
  des polygraphes (c'est le cas de $\On{n}$), les $p$\=/chaînes de
  $\lambda(C)$ sont librement engendrées par les générateurs de dimension $p$
  de $C$ et la différentielle est donnée par le but moins la source.

  Ainsi, le complexe dirigé augmenté $\lambda(\On{n})$ se décrit de la
  manière suivante :
  \begin{itemize}
    \item Le complexe de chaînes sous-jacent est le complexe de chaînes
    normalisé associé à l'ensemble simplicial $\Deltan{n}$. Ainsi,
    $\lambda(\On{n})_p$ est le groupe abélien libre sur l'ensemble
    \[
    \{(i_0, \dots, i_p) \mid 0 \le i_0 < \cdots < i_p \le n\}
    \]
    et la différentielle, pour $p \ge 1$, est définie par
    \[ d(i_0, \dots, i_p) = \sum_{l = 0}^p (-1)^l (i_0, \dots, \hat{i_l},
    \dots, i_p), \]
    où $(i_0, \dots, \hat{i_l}, \dots, i_p) = (i_0, \dots, i_{l-1}, i_{l+1},
    \dots, i_p)$.
    \item L'augmentation $e : \lambda(\On{n})_0 \to \Z$ envoie $(i_0)$ sur
    $1$.
    \item Enfin, le sous-monoïde $\lambda(\On{n})_p^\ast$ consiste en
    l'ensemble des $p$-chaînes dont les coefficients sont positifs (au sens
    large).
  \end{itemize}
  Quant au morphisme $\lambda(sr) : \lambda(\On{n}) \to \lambda(\On{n})$,
  puisque le \oo-foncteur $sr$ est l'identité sur les objets, envoie
  pour tout $0 \le i < j \le n$ la flèche « directe » $(i, j)$ de $i$
  vers~$j$ dans $\On{n}$ sur le composé $(j-1, j) (j-2, j-1) \cdots (i,
  i+1)$ et envoie toute cellule de dimension au moins $2$ sur une identité,
  on a
  \[ \lambda(sr)(i_0, \dots, i_p) =
  \begin{cases}
    (i_0) & \text{si $p = 0$,} \\
    \sum_{i_0 < k \le i_1} (k-1, k) & \text{si $p = 1$,} \\
     0 & \text{si $p \ge 2$.}
  \end{cases}
  \]
  En vertu de \cite[proposition B.4.7 et théorème 2.11]{AraMaltsiJoint}, la
  donnée d'une transformation oplax de $\id{\On{n}}$ vers $sr$ est
  équivalente à celle d'une homotopie de $\id{\lambda(\On{n})}$ vers
  $\lambda(sr)$, une homotopie entre morphismes de complexes dirigés
  augmentés étant simplement une homotopie entre les morphismes de complexes
  de chaînes sous-jacents qui envoie les chaînes positives sur des chaînes
  positives.  Il nous suffit donc de définir une telle homotopie.

  Posons
  \[ h(i_0, \dots, i_p) =
  \begin{cases}
     0 & \text{si $p = 0$,} \\
     \sum_{i_0 < k < i_1} (k-1, k, i_1, \dots, i_p) & \text{si $p \ge 1$.}
  \end{cases}
  \]
  Par définition, $h$ envoie bien les chaînes positives sur des chaînes
  positives et il s'agit de vérifier que $h$ est une homotopie de complexes
  de chaînes.
  \begin{itemize}
    \item Si $p = 0$, on a
    \[ dh(i_0) = d(0) = 0 = (i_0) - (i_0) = \lambda(sr)(i_0) -
    \id{\lambda(\On{n})}(i_0). \]
    \item Si $p = 1$, on a
    {
      \allowdisplaybreaks
      \begin{align*}
        dh(i_0, i_1) + hd(i_0, i_1) & =
        \sum_{i_0 < k < i_1} d(k-1, k, i_1) + h(i_1) - h(i_0) \\
        &  =
        \sum_{i_0 < k < i_1} (k, i_1)
        -
        \sum_{i_0 < k < i_1} (k-1, i_1)
        +
        \sum_{i_0 < k < i_1} (k-1, k) \\*
        &  =
        (i_1 - 1, i_1)
        -
        (i_0, i_1)
        +
        \sum_{i_0 < k < i_1} (k-1, k) \\
        &  =
        \sum_{i_0 < k \le i_1} (k-1, k)
        -
        (i_0, i_1) \\*
        & =
        \lambda(sr)(i_0, i_1) - \id{\lambda(\On{n})}(i_0, i_1).
      \end{align*}
    }
    \item Enfin, si $p \ge 2$, on a
    \[
      \begin{split}
        dh(i_0, \dots, i_p)
        & = \sum_{i_0 < k < i_1} d(k-1, k, i_1, \dots, i_p) \\
        & = 
        \sum_{i_0 < k < i_1} (k, i_1, \dots, i_p)
        -
        \sum_{i_0 < k < i_1} (k-1, i_1, \dots, i_p) \\*
        & \phantom{=1}\qquad
        +
        \sum_{i_0 < k < i_1} \sum_{l=1}^p (-1)^{l+1} (k-1, k, i_1, \dots, \hat{i_l}, \dots, i_p) \\
        & = 
        (i_1-1, i_1, \dots, i_p)
        -
        (i_0, i_1, \dots, i_p) \\*
        & \phantom{=1}\qquad
        +
        \sum_{i_0 < k < i_1} \sum_{l=1}^p (-1)^{l+1} (k-1, k, i_1, \dots, \hat{i_l}, \dots, i_p) \\
      \end{split}
    \]
    et
    \[
      \begin{split}
        hd(i_0, \dots, i_p)
        & =
        \sum_{l=0}^p (-1)^l h(i_0, \dots, \hat{i_l}, \dots, i_p) \\
        & =
        \sum_{i_1 < k < i_2}^p (k-1, k, i_2, \dots, i_p)
        -
        \sum_{i_0 < k < i_2}^p (k-1, k, i_2, \dots, i_p) \\*
        & \phantom{=1}\qquad
        +
        \sum_{l=2}^p (-1)^l \sum_{i_0 < k < i_1} (k-1, k, i_1, \dots, \hat{i_l}, \dots, i_p) \\
        & =
        -
        \sum_{i_0 < k \le i_1}^p (k-1, k, i_2, \dots, i_p) \\*
        & \phantom{=1}\qquad
        +
        \sum_{l=2}^p (-1)^l \sum_{i_0 < k < i_1} (k-1, k, i_1, \dots, \hat{i_l}, \dots, i_p) \\
        & =
        -
        (i_1-1, i_1, i_2, \dots, i_p) \\*
        & \phantom{=1}\qquad
        +
        \sum_{l=1}^p (-1)^l \sum_{i_0 < k < i_1} (k-1, k, i_1, \dots,
        \hat{i_l}, \dots, i_p), \\
      \end{split}
    \]
    d'où
    \[
    \begin{split}
    dh(i_0, \dots, i_p) + hd(i_0, \dots, i_p) & = -(i_0, \dots, i_p) \\
    & = \lambda(sr)(i_0, \dots, i_p) - \id{\lambda(\On{n})}(i_0, \dots, i_p).
    \end{split}
    \]
  \end{itemize}
  Ceci achève de montrer que $h$ est bien une homotopie et démontre donc
  l'assertion.
\end{proof}

\bibliography{biblio}
\bibliographystyle{mysmfplain}

\end{document}